\documentclass[a4paper,11pt]{article}
\usepackage{amsfonts, amsmath, amssymb, amsthm, verbatim}
\usepackage{color}

\newtheorem{theorem}{Theorem}
\newtheorem{proposition}{Proposition}
\newtheorem{lemma}{Lemma}
\newtheorem{corollary}{Corollary}

\theoremstyle{definition}
\newtheorem{definition}{Definition}
\newtheorem{remark}{Remark}
\newtheorem{assump}{Assumption}

\newcommand{\R}{\mathbb R}
\newcommand{\N}{\mathbb N}
\newcommand{\OO}{\mathcal O}
\newcommand{\F}{\mathcal F}
\newcommand{\LL}{\mathcal L}

\newcommand{\f}{{\tilde f}}

\newcommand{\V}{\mathcal{V}}
\newcommand{\eps}{\varepsilon}
\newcommand{\tauu}{\bar{\tau}}

\DeclareMathOperator{\Ric}{Ric}
\DeclareMathOperator{\Rm}{Rm}

\DeclareMathOperator{\Hess}{Hess}
\DeclareMathOperator{\Id}{Id}
\DeclareMathOperator{\dive}{div}

\DeclareMathOperator{\vol}{vol}

\renewcommand{\a}{\alpha}
\newcommand{\gm}{\gamma}
\newcommand{\dl}{\delta}
\newcommand{\zt}{\zeta}
\newcommand{\sg}{\sigma}
\newcommand{\e}{\mathrm{e}}
\newcommand{\E}{\mathbb{E}}
\renewcommand{\P}{\mathbb{P}}

\newcommand{\abra}[1]{\left( #1 \right)}
\newcommand{\bbra}[1]{\left\{ #1 \right\}}
\newcommand{\cbra}[1]{\left[ #1 \right]}
\newcommand{\dbra}[1]{\langle #1 \rangle}
\newcommand{\ebra}[1]{\lfloor #1 \rfloor}

\DeclareMathOperator{\Lcut}{\LL Cut}

\author{
Kazumasa Kuwada\footnote{Graduate School of Humanities and Sciences, Ochanomizu University, Tokyo 112-8610, Japan.
\break 
E-mail: \texttt{kuwada.kazumasa@ocha.ac.jp} /
Present address: Institut f\"{u}r Angewandte Mathematik, Universit\"{a}t Bonn, Endenicher Allee 60, 53115 Bonn, Germany.
E-mail: \texttt{kuwada@iam.uni-bonn.de}
}
\footnote{Partially supported by the JSPS fellowship for research abroad.}
\hspace*{0em}
and Robert Philipowski\footnote{Institut f\"{u}r Angewandte Mathematik, Universit\"{a}t Bonn, Endenicher Allee 60, 53115 Bonn, Germany.
\break E-mail: \texttt{philipowski@iam.uni-bonn.de}
}
}

\title{Coupling of Brownian motions and Perelman's $\LL$-functional}
\date{}

\begin{document}
\maketitle

\begin{abstract}
We show that on a manifold whose Riemannian metric evolves under backwards Ricci flow two Brownian motions can be coupled in such a way
that the expectation of their normalized $\LL$-distance is non-increasing.
As an immediate corollary we obtain a new proof of a recent result of Topping (J. reine angew. Math. 636 (2009), 93--122),
namely that the normalized $\LL$-transportation cost between two solutions of the heat equation is non-increasing as well.
\\

\end{abstract}

\section{Introduction}
Let $M$ be a $d$-dimensional differentiable manifold, 
$0 \le \tauu_1 < \tauu_2 < T$ and $(g(\tau))_{\tau \in [\tauu_1, T]}$ a complete 
backwards Ricci flow on $M$, i.e.~a smooth family of Riemannian metrics satisfying
\begin{equation} \label{eq:Rf}
\frac{\partial g}{\partial \tau} = 2 \Ric_{g (\tau)}
\end{equation}
and such that $(M, g(\tau))$ is complete for all $\tau \in [\tauu_1, T]$.
In this situation Perelman~\cite[Section~7.1]{perelman1} (see also \cite[Definition~7.5]{chowetal}) defined the $\LL$-functional of a smooth curve $\gamma: [\tau_1, \tau_2] \to M$
(where $\tauu_1 \leq \tau_1 < \tau_2 \leq T$) by
\[
\LL(\gamma) := \int_{\tau_1}^{\tau_2} \sqrt{\tau} \left[ \left| \dot{\gamma}(\tau) \right|_{g(\tau)}^2 + R_{g(\tau)}(\gamma(\tau)) \right] d\tau,
\]
where $R_{g(\tau)}(x)$ is the scalar curvature at $x$ with respect to the metric $g(\tau)$.

Denoting by $L(x, \tau_1; y, \tau_2)$ the infimum of $\LL(\gamma)$ over smooth curves
$\gamma: [\tau_1, \tau_2] \to M$ satisfying $\gamma(\tau_1) = x$ and $\gamma(\tau_2) = y$, and by
\[
W^\LL(\mu, \tau_1; \nu, \tau_2) := \inf_\pi \int_{M \times M} L(x, \tau_1; y, \tau_2) \pi(dx, dy)
\]
(the infimum is over all probability measures $\pi$ on $M \times M$ whose marginals are $\mu$ and $\nu$)
the associated transportation cost between two probability measures $\mu$ and $\nu$ on $M$,
Topping~\cite{topping} (see also Lott~\cite{lott}) 
obtained the following result: 
\begin{theorem}[Theorem~1.1 in \cite{topping}]\label{topp}
Assume that $M$ is compact and that $\tauu_1 > 0$.
Let
$u: [\tauu_1, T] \times M \to \R_+$ 
and
$v: [\tauu_2, T] \times M \to \R_+$ 
be two non-negative unit-mass solutions of the heat equation
\[
\frac{\partial u}{\partial \tau} = \Delta_{g(\tau)} u - R u, 
\]
where the term $R u$ comes from the change in time of the volume element.  
Then the normalized $\LL$-transportation cost
\[
\bar{\Theta}(t) := 2 (\sqrt{\tauu_2 t} - \sqrt{\tauu_1 t}) W^\LL(u(\tauu_1 t, \cdot) \vol_{g(\tauu_1 t)}, \tauu_1 t; v(\tauu_2 t, \cdot) \vol_{g(\tauu_2 t)}, \tauu_2 t)
- 2 d \left( \sqrt{\tauu_2 t} - \sqrt{\tauu_1 t} \right)^2
\]
between the two solutions evaluated at times $\tauu_1 t$ resp.\ $\tauu_2 t$ is a non-increasing function of $t \in [1, T/\tauu_2]$.
\end{theorem}

By $g(\tau)$-Brownian motion, 
we mean 
the time-inhomogeneous diffusion process 
whose generator is $\Delta_{g(\tau)}$. 
As in the time-homogeneous case, 
the heat distribution 
$u ( \tau , \cdot ) \vol_{g(\tau)}$ is expressed 
as the law of a $g(\tau)$-Brownian motion 
at time $\tau$. 
In view of this strong relation between heat equation and Brownian motion, it is natural to ask whether one can couple two Brownian motions on $M$
in such a way that a pathwise analogue of this result involving the function
\[
\Theta(t,x,y) := 2 \left( \sqrt{\tauu_2 t} - \sqrt{\tauu_1 t} \right) L(x, \tauu_1 t; y, \tauu_2 t) - 2 d \left( \sqrt{\tauu_2 t} - \sqrt{\tauu_1 t} \right)^2.
\]
holds. 
The main result of this paper answers it affirmatively 
as follows: 
\begin{theorem}\label{mainresult}
Assume that $M$ has bounded curvature tensor, 
i.e.
\begin{equation} \label{eq:kb}
\sup_{x \in M, \tau \in [\tauu_1, T]} |\Rm_{g(\tau)}|_{g(\tau)}(x) < \infty.
\end{equation}
Then given any points $x,y \in M$ and any $s \in [1, T/\tauu_2]$,
there exist two coupled $g(\tau)$-Brownian motions
$(X_\tau)_{\tau \in [\tauu_1 s, T]}$
and
$(Y_\tau)_{\tau \in [\tauu_2 s , T ]}$
with initial values $X_{\tauu_1 s} = x$ and $Y_{\tauu_2 s} = y$ 
such that
the process $(\Theta(t, X_{\tauu_1 t}, Y_{\tauu_2 t} ))_{t \in [s, T/\tauu_2]}$ 
is a supermartingale.
In particular $\E \left[ \Theta(t, X_{\tauu_1 t}, Y_{\tauu_2 t} ) \right]$ is non-increasing.
In addition, we can take them so that
the map $(x,y) \mapsto (X,Y)$ is measurable.
\end{theorem}

\begin{remark}
Obviously, \eqref{eq:kb} is satisfied if $M$ is compact.
Thus it includes the case of Theorem~\ref{topp}.
In addition,
there are plenty of examples of backwards Ricci flow 
satisfying \eqref{eq:kb} even when $M$ is non-compact. 
Indeed, 
given a metric $g_0$ on $M$ with bounded curvature tensor, 
there exists a unique solution to the Ricci flow 
$\partial_t g (t) = - 2 \Ric_{g(t)}$ 
with initial condition $g_0$ 
satisfying \eqref{eq:kb} 
for a short time 
(see \cite{shi} for existence and 
\cite{chenzhu} for uniqueness). 
Then the corresponding backwards Ricci flow 
is obtained by time-reversal.  
\end{remark}

\begin{remark}
As shown in \cite{kuwadaphilipowski}, under backwards Ricci flow $g(\tau)$-Brownian motion cannot explode. Hence $\Theta(t,X_t,Y_t)$ is well-defined for all $t \in [s, T/\tauu_2]$. 
This fact also ensures that $u (\tau , \cdot) \vol_{g(\tau)}$ 
has unit mass whenever it does at the initial time. 
\end{remark}

Using Theorem~\ref{mainresult} we can prove 
Topping's result even in the non-compact case. 

\begin{theorem} \label{topp1}
Assume that \eqref{eq:kb} holds. 
Then the same assertion as in Theorem~\ref{topp} 
holds true 
for nonegative unit mass solutions $u$ and $v$ 
to the heat equation 
and the associated functional $\bar{\Theta} (t)$. 
\end{theorem}

\begin{proof}[Proof of Theorem~\ref{topp1} using Theorem~\ref{mainresult}]
Fix $1 \leq s < t \leq T / \tauu_2$, and let $\pi$ be an optimal coupling of $u(\tauu_1 s, \cdot) \vol_{g(\tauu_1 s)}$ and $v(\tauu_2 s, \cdot) \vol_{g(\tauu_2 s)}$.
(Existence of an optimal coupling follows from \cite[Theorem~4.1]{villani}, using the obvious lower bound
$L(x, \tau_1; y, \tau_2) \geq \frac{2}{3} (\tau_2^{3/2} - \tau_1^{3/2} ) \inf_{x \in M, \tau \in [\tauu_1, T]} R_{g(\tau)}(x)$.)
For each $(x,y) \in M \times M$, we take coupled Brownian motions
$(X_{\tau}^x)_{\tau \in [\tauu_1 s, T]}$ and $(Y_{\tau}^y)_{\tau \in [\tauu_2 s, T]}$
with initial values $X_{\tauu_1 s}^x = x$ and $Y_{\tauu_2 s}^y = y$ 
as in Theorem~\ref{mainresult}. 
Since $(x,y) \mapsto (X^x , Y^y)$ is measurable, 
we can construct a coupling of two Brownian motions 
$(X, Y)$ with initial distribution $\pi$ 
by following a usual manner. 
%
%
Then 
the joint distribution of $X_{\tauu_1 t}$ and $Y_{\tauu_2 t}$ is a coupling of $u(\tauu_1 t, \cdot) \vol_{g(\tauu_1 t)}$ and $v(\tauu_2 t, \cdot) \vol_{g(\tauu_2 t)}$, so that
$\bar\Theta(t) \leq \E \left[ \Theta(t, X_{\tauu_1 t} , Y_{\tauu_2 t} ) \right] \leq \E \left[ \Theta(s, X_{\tauu_1 s} , Y_{\tauu_2 s} ) \right] = \bar\Theta(s)$.
\end{proof}

\section{Remarks concerning related work}

The Ricci flow was introduced by Hamilton~\cite{hamilton}. 
There he effectively used it 
to solve the Poincar\'{e} conjecture 
for 3-manifolds with positive Ricci curvature. 
By following his approach, 
Perelman~\cite{perelman1, perelman2, perelman3} 
finally solved the Poincar\'{e} conjecture 
(see also \cite{caozhu,kleinerlott, morgantian}). 
There he used $\LL$-functional as a crucial tool. 
At the same stage,  
he also studied the heat equation in \cite{perelman1} 
in relation with the geometry of Ricci flows. 
It suggests that analysing the heat equation is 
still an efficient way 
to investigate geometry of the underlying space 
even in the time-dependent metric case. 
This general principle has been confirmed 
in recent developments in this direction. 
In connection with the theory of optimal transportation, 
McCann and Topping~\cite{mccanntopping} showed contraction in the $L^2$-Wasserstein distance 
for the heat equation under backwards Ricci flow on a compact manifold. 
Topping's result~\cite{topping} can be regarded as 
an extension of it 
to contraction in the normalized $\LL$-transportation cost 
(see \cite{lott} also). 
By taking $\tauu_2 \to \tauu_1$, 
he recovered the monotonicity of Perelman's $\mathcal{W}$-entropy, 
which is one of fundamental ingredients in Perelman's work. 

A probabilistic approach to these problems is initiated 
by Arnaudon, Coulibaly and Thalmaier. 
In \cite[Section 4]{act2}, they 
sharpened McCann and Topping's result \cite{mccanntopping} 
to a pathwise contraction in the following sense:
There is a coupling $(X_t, Y_t)_{t \geq 0}$ of two Brownian motions starting from $x, y \in X$ respectively 
such that the $g(t)$-distance between $X_t$ and $Y_t$ is non-increasing in $t$ almost surely. 
In their approach, probabilistic techniques 
based on analysis of sample paths 
made it possible to establish such a pathwise estimate. 
It should be mentioned that, 
as another advantage of their approach, 
their argument works even on non-compact $M$ (cf.~\cite{kuwadaphilipowski}). 
Our approach is the same as theirs in spirit. 
In fact, such advantages are also inherited to our results. 
Unfortunately, 
we cannot expect a pathwise contraction as theirs 
since our problem differs in nature 
from what is studied in \cite{act} (see Remark~\ref{rem:mart}). 
However, it should be noted that 
this new fact is revealed 
as a result of our pathwise arguments. 
Furthermore, we can expect that 
our approach makes it possible 
to employ several techniques 
in stochastic analysis 
to obtain more detailed behavior of 
$\Theta (t , X_{\tauu_1 t} , Y_{\tauu_2 t} )$, 
especially in the limit $\tauu_2 \to \tauu_1$, 
in a future development. 
Note that, from technical point of view, 
our method relies on the result in \cite{kuwada2} and 
it is different from Arnaudon, Coulibaly and Thalmaier's one.

\section{Coupling of Brownian motions in the absence of $\LL$-cut locus}
\label{sec:SDE}
%
Since the proof of Theorem~\ref{mainresult} 
involves some technical arguments, 
first we study the problem in the case that 
the $\LL$-distance $L$ has no singularity. 
More precisely,
\begin{assump} \label{ass:empty} 
The $\LL$-cut locus is empty. 
\end{assump}
See subsection~\ref{sec:geom} or \cite{chowetal,topping,ye} 
for the definition of $\LL$-cut locus. 
Under Assumption~\ref{ass:empty}, 
the following holds: 
\begin{enumerate}
\item For all $x,y \in M$ and all $\tauu_1 \leq \tau_1 < \tau_2 \leq T$ there is a 
unique minimizer $\gamma_{xy}^{\tau_1 \tau_2}$ of 
$L ( x, \tau_1 ; y , \tau_2 )$
(existence of $\gamma_{xy}^{\tau_1 \tau_2}$ is proved in \cite[Lemma~7.27]{chowetal}, while uniqueness
follows immediately from the characterization of $\LL$-cut locus, 
see subsection \ref{sec:geom}). 
\item The function $L$ is globally smooth.
\end{enumerate}
Thus, in this case, 
we can freely use stochastic analysis on the frame bundle 
without taking any care on regularity of $L$. 
In section~\ref{randomwalks}, 
we present the complete proof of Theorem~\ref{mainresult} 
using a random walk approximation 
(see Remark~\ref{rem:gRW} for further details on the choice of our approach). 

\subsection{Construction of the coupling}
A $g(\tau)$-Brownian motion $\tilde{X}$ on $M$ (scaled in time by the factor $\tauu_1$) starting at a point $x \in M$ at time $s \in [1, T/\tauu_2]$
can be constructed in the following way \cite{act, coulibaly, kuwadaphilipowski}:
Let $\pi: \F(M) \to M$ be the frame bundle and $(e_i)_{i = 1}^d$ the standard basis of $\R^d$.
For each $\tau \in [\tauu_1, T]$ let $(H_i(\tau))_{i=1}^d$ be the associated $g(\tau)$-horizontal vector fields on $\F(M)$ (i.e.~$H_i(\tau, u)$ is the $g(\tau)$-horizontal lift of $u e_i$).
Moreover let $(\V^{\alpha, \beta})_{\alpha, \beta = 1}^d$ be the canonical vertical vector fields,
i.e.~$(\V^{\alpha, \beta} f)(u) := \left. \frac{\partial}{\partial m_{\alpha \beta}} \right|_{\mathbf{m} = \Id} \left( f(u(\mathbf{m})) \right)$
($\mathbf{m} = (m_{\alpha \beta})_{\alpha, \beta = 1}^d \in GL_d(\R)$),
and let $(W_t)_{t \geq 0}$ be a standard $\R^d$-valued Brownian motion.
By $\OO^{g(\tau)} (M)$, we denote the $g(\tau)$-orthonormal frame bundle. 

We first define 
a horizontal Brownian motion on $ \F(M)$ 
as the solution 
$\tilde{U} = (\tilde{U}_t)_{t \in [s, T/\tauu_1]}$ 
of the Stratonovich SDE
\begin{equation}\label{gtbm}
d\tilde{U}_t 
= 
\sqrt{2 \tauu_1} \sum_{i=1}^d H_i(\tauu_1 t, \tilde{U}_t) \circ dW_t^i
- 
\tauu_1 \sum_{\alpha, \beta = 1}^d 
\frac{ \partial g }{\partial \tau} (\tauu_1 t) 
(\tilde{U}_t e_\alpha, \tilde{U}_t e_\beta) 
\V^{\alpha \beta}(\tilde{U}_t) dt
\end{equation}
with initial value 
$\tilde{U}_s = u \in \OO^{g(\tauu_1 s)}_x (M)$, 
and then define a scaled Brownian motion $\tilde{X}$ on $M$ as
\[
\tilde{X}_t := \pi \tilde{U}_t.
\]
Note that $\tilde{X}_t$ does not move when $\tauu_1  =0$.
The last term in \eqref{gtbm} ensures that 
$\tilde{U}_t \in \OO^{g(\tauu_1 t)}(M)$ for all $t \in [s, T/\tauu_1]$
(see \cite[Proposition~1.1]{act}, \cite[Proposition~1.2]{coulibaly}), 
so that 
by It\^o's formula for all smooth $f: [s, T/\tauu_1] \times M \to \R$
\[
df(t, \tilde{X}_t) 
= 
\frac{\partial f}{\partial t}(t, \tilde{X}_t) dt 
+ 
\sqrt{2 \tauu_1} \sum_{i = 1}^d (\tilde{U}_t e_i) f(t, \tilde{X}_t) d W_t^i 
+ 
\tauu_1 \Delta_{g(\tau_1 t)} f(t, \tilde{X}_t) dt.
\]
Let us define
$( X_\tau )_{\tau \in [ \tauu_1 s, T ]}$
by $X_{\tauu_1 t} : = \tilde{X}_t$.
Then $X_\tau$ becomes a $g(\tau)$-Brownian motion
when $\tauu_1 > 0$.
\begin{remark} \label{rem:time}
Intuitively, it might be helpful to think that 
$X_\tau$ lives in $( M , g(\tau))$, or 
$\tilde{X}_t$ lives in $( M , g (\tauu_1 t) )$. 
The same is true for $Y$ and $\tilde{Y}$ 
which will be defined below. 
Similarly, 
for all curves $\gamma : [ \tau_1 ,\tau_2 ] \to M$ 
in this paper, we can naturally regard 
$\gm (\tau)$ as in $(M , g(\tau) )$. 
\end{remark}
We now want to construct 
a second scaled Brownian motion $\tilde{Y}$ on $M$ 
in such a way that 
its infinitesimal increments $d \tilde{Y}_t$ are 
``space-time parallel'' to those of $\tilde{X}$
along the minimal $\LL$-geodesic (namely, the minimizer of $L$) 
from $( \tilde{X}_t , \tauu_1 t )$ to $( \tilde{Y}_t , \tauu_2 t )$.
To make this idea precise, 
we first define the notion of space-time parallel vector field:

\begin{definition}[space-time parallel vector field]
Let $\tauu_1 \leq \tau_1 < \tau_2 \leq T$ and $\gamma: [\tau_1, \tau_2] \to M$ be a smooth curve. We say that a vector field $Z$ along $\gamma$ is {\em space-time parallel} if
\begin{equation}\label{lparallel}
\nabla_{\dot{\gm}(\tau)}^{g(\tau)} Z(\tau) 
= 
-\Ric_{g(\tau)}^\# ( Z(\tau) )
\end{equation}
holds for all $\tau \in [\tau_1, \tau_2]$. 
Here $\nabla^{g(\tau)}$ stands for 
the covariant derivative associated with 
the $g(\tau)$-Levi-Civita connection and 
$\Ric_{g(\tau)}^\#$ is defined 
by regarding the $g(\tau)$-Ricci curvature as a (1,1)-tensor. 
\end{definition}

\begin{remark}
Since \eqref{lparallel} is a linear first-order ODE, for any $\xi \in T_{\gamma(\tau_1)}M$ there exists a unique space-time parallel vector field $Z$ along $\gamma$ with $Z (\tau_1) = \xi$.
\end{remark}

\begin{remark} \label{rem:isom}
Whenever $Z$ and $Z'$ are space-time parallel vector fields along a curve $\gamma$, their $g(\tau)$-inner product is constant in $\tau$:
\begin{align*}
\frac{d}{d \tau} \langle Z(\tau), Z'(\tau) \rangle_{g(\tau)}
& = 
\frac{\partial g }{\partial \tau}(\tau) (Z(\tau), Z'(\tau)) 
+ 
\langle 
  \nabla_{\dot{\gm} (\tau)}^{g(\tau)} Z(\tau)
  , 
  Z'(\tau) 
\rangle_{g(\tau)} 
+ 
\langle 
  Z(\tau)
  , 
  \nabla_{\dot{\gm} (\tau)}^{g(\tau)} Z'(\tau) 
\rangle_{g(\tau)} \\
& = 
2 \Ric_{g(\tau)} (Z(\tau), Z'(\tau)) 
- \Ric_{g(\tau)} (Z(\tau), Z'(\tau)) 
- \Ric_{g(\tau)} (Z(\tau), Z'(\tau)) 
\\
& = 0.
\end{align*}
\end{remark}
\begin{remark}
The emergence of the Ricci curvature 
in \eqref{lparallel} is based 
on the Ricci flow equation \eqref{eq:Rf}. 
Indeed, we can generalize 
the notion of space-time parallel transport 
even in the absence of \eqref{eq:Rf} 
with keeping the property in the last remark. 
This would be a natural extension in the sense that 
it coincides with the usual parallel transport 
when $g(\tau)$ is constant in $\tau$. 
On the other hand, it is convenient 
to define it as \eqref{lparallel}
for later use in this paper. 
\end{remark}

\begin{definition}[space-time parallel transport] \label{def:pt}
For $x,y \in M$ and $\tauu_1 \leq \tau_1 < \tau_2 \leq T$,
we define a map $m_{xy}^{\tau_1 \tau_2}: T_x M \to T_y M$ as follows: $m_{xy}^{\tau_1 \tau_2}(\xi) := Z(\tau_2)$,
where $Z$ is the unique space-time parallel vector field along $\gamma_{xy}^{\tau_1 \tau_2}$ with $Z (\tau_1) = \xi$.
By Remark~\ref{rem:isom}, 
$m_{xy}^{\tau_1 \tau_2}$ is an isometry 
from $( T_x M , g ( \tau_1 ) )$ 
to $( T_y M , g ( \tau_2 ) )$. 
In addition, it smoothly depends on $x, \tau_1 , y, \tau_2$ 
under Assumption~\ref{ass:empty}. 
\end{definition}

We now define a second horizontal scaled Brownian motion 
$\tilde{V} = (\tilde{V}_t)_{t \in [s, T/\tauu_2]}$ on $\F(M)$ 
as the solution of
\[
d\tilde{V}_t 
= 
\sqrt{2 \tauu_2} 
\sum_{i=1}^d H_i^*( \tilde{U}_t , \tauu_1 t ; \tilde{V}_t , \tauu_2 t ) 
\circ dW_t^i
- \tauu_2 
\sum_{\alpha, \beta = 1}^d 
\frac{\partial g}{\partial \tau}(\tauu_2 t)
(\tilde{V}_t e_\alpha, \tilde{V}_t e_\beta) 
\V^{\alpha \beta}(\tilde{V}_t) dt
\]
with initial value $\tilde{V}_s = v \in \OO^{g(\tauu_2 s)}_y (M)$, 
and we set $\tilde{Y}_t := \pi \tilde{V}_t$.
Here 
$H_i^*(u , \tau_1 ; v , \tau_2)$ is 
the $g(\tau_2)$-horizontal lift of $v e_i^*(u , \tau_1 ; v , \tau_2 )$, 
where
\[
e_i^* ( u , \tau_1 ; v , \tau_2 ) 
:= 
v^{-1} m_{\pi u, \pi v}^{\tau_1, \tau_2} u e_i.
\]
As we did for $\tilde{X}$, 
let us define 
$( Y_\tau )_{\tau \in [ \tauu_2 s , T ]}$ 
by $Y_{\tauu_2 t} := \tilde{Y}_t$ 
to make $Y$ a $g(\tau)$-Brownian motion. 
From theoretical point of view, 
it seems to be natural to work with $( X_\tau , Y_\tau )$ 
(see Remark~\ref{rem:time}). 
However, for technical simplicity, 
we will prefer to work with $( \tilde{X}_t , \tilde{Y}_t )$ instead 
in the sequel. 

\subsection{Proof of Theorem~\ref{mainresult} in the absence of $\LL$-cut locus}
Our argument in this section is based on 
the following It\^o's formula for $( \tilde{X}_t , \tilde{Y}_t )$. 

\begin{lemma}\label{itoformula}
Let $f$ be a smooth function on $[s, T / \tauu_2] \times M \times M$. Then
\begin{multline*}
d f(t, \tilde{X}_t, \tilde{Y}_t) 
= 
\frac{\partial f}{\partial t}(t, \tilde{X}_t, \tilde{Y}_t) dt 
+ 
\sum_{i = 1}^d 
\left[ 
  \sqrt{2 \tauu_1} \tilde{U}_t e_i 
  \oplus 
  \sqrt{2 \tauu_2} \tilde{V}_t e_i^* 
\right] 
f(t, \tilde{X}_t, \tilde{Y}_t) d W_t^i\\
+ 
\sum_{i=1}^d 
\left. 
  \Hess_{g(\tauu_1 t) \oplus g(\tauu_2 t)} f 
\right|_{(t,\tilde{X}_t,\tilde{Y}_t)}
\left( 
  \sqrt{\tauu_1} \tilde{U}_t e_i 
  \oplus 
  \sqrt{\tauu_2} \tilde{V}_t e_i^*, 
  \sqrt{\tauu_1} \tilde{U}_t e_i 
  \oplus 
  \sqrt{\tauu_2} \tilde{V}_t e_i^* 
\right) dt.
\end{multline*}
Here the Hessian of $f$ is taken with respect to the product metric $g(\tauu_1 t) \oplus g(\tauu_2 t)$, 
$e_i^*$ stands for 
$e_i^*( \tilde{U}_t , \tauu_1 t ; \tilde{V}_t , \tauu_2 t )$,
and for tangent vectors $\xi_1 \in T_x M$, $\xi_2 \in T_y M$ we write $\xi_1 \oplus \xi_2 := (\xi_1, \xi_2) \in T_{(x,y)} (M \times M)$.
\end{lemma}

\begin{proof}
It\^o's formula applied to a smooth function $\f$ on $[s, T / \tauu_2] \times \F(M) \times \F(M)$ gives
\begin{align*}
d \f & (t, \tilde{U}_t, \tilde{V}_t) \\
& = 
\frac{\partial \f}{\partial t}(t, \tilde{U}_t, \tilde{V}_t) dt
+ 
\sum_{i = 1}^d 
\left[ 
  \sqrt{2 \tauu_1} H_i(\tauu_1 t, \tilde{U}_t) 
  \oplus 
  \sqrt{2 \tauu_2} H_i^*(\tilde{U}_t, \tauu_1 t ; \tilde{V}_t , \tauu_2 t) 
\right] \f(t,\tilde{U}_t,\tilde{V}_t) d W_t^i\\
& \quad 
+ 
\sum_{i=1}^d 
\left[ 
  \sqrt{\tauu_1} H_i(\tauu_1 t, \tilde{U}_t) 
  \oplus 
  \sqrt{\tauu_2} H_i^*(\tilde{U}_t, \tauu_1 t ; \tilde{V}_t , \tauu_2 t) 
\right]^2 
\f(t,\tilde{U}_t,\tilde{V}_t) dt\\
& \quad 
- 
\sum_{\alpha, \beta = 1}^d 
\left[ 
  \tauu_1 
  \frac{\partial g}{\partial \tau} ( \tauu_1 t )
  (\tilde{U}_t e_\alpha, \tilde{U}_t e_\beta) 
  \V^{\alpha \beta}(\tilde{U}_t)
  \oplus 
  \tauu_2 
  \frac{\partial g}{\partial \tau} ( \tauu_2 t ) 
  (\tilde{V}_t e_\alpha, \tilde{V}_t e_\beta) 
  \V^{\alpha \beta}(\tilde{V}_t) 
\right] 
\f (t, \tilde{U}_t, \tilde{V}_t) dt.
\end{align*}
The claim follows by choosing $\f(t,u,v) := f(t, \pi u, \pi v)$ because the function considered here
is constant in the vertical direction so that the term involving $\V^{\alpha \beta} \f$ vanishes.
\end{proof}

Let $\Lambda (t,x,y) := L(x, \tauu_1 t; y, \tauu_2 t)$. 
In order to apply Lemma~\ref{itoformula} to the function $\Theta$ 
we need the following proposition whose proof is given in the next section:
\begin{proposition}\label{lambda}
Take $x,y \in M$, 
$u \in \OO^{g(\tauu_1 t)}_x (M)$ and $v \in \OO^{g(\tauu_2 t)}_y (M)$. 
Let $\gamma$ be a minimizer of $L ( x, \tauu_1 t ; y, \tauu_2 t )$. 
Assume that 
$( x, \tauu_1 t ; y, \tauu_2 t )$ is not 
in the $\LL$-cut locus. 
Set 
$
\xi_i 
:= 
\sqrt{\tauu_1} u e_i 
\oplus 
\sqrt{\tauu_2} v e_i^*( u, \tauu_1 t ; v , \tauu_2 t )
$. 
Then
\begin{align}
\frac{\partial \Lambda}{\partial t}(t,x,y) 
& = 
\frac{1}{t} \int_{\tauu_1 t}^{\tauu_2 t} \tau^{3/2}
\bigg( \frac{3}{2 \tau} R_{g(\tau)}(\gamma(\tau)) - \Delta_{g(\tau)} R_{g(\tau)}(\gamma(\tau))
- 2 \, | \Ric_{g(\tau)} |_{g(\tau)}^2 (\gamma(\tau)) \nonumber\\
& \hspace{8em} 
- \frac{1}{2 \tau} |\dot \gamma(\tau)|_{g(\tau)}^2 + 2 \Ric_{g(\tau)}(\dot\gamma(\tau), \dot\gamma(\tau)) \bigg) d \tau \label{dlambdadt},
\\ \nonumber
\sum_{i=1}^d \Hess_{g(\tau_1) \oplus g(\tau_2)} & \Lambda \Big|_{(t, x, y)} (\xi_i, \xi_i) \nonumber\\
& \hspace{-5em} \leq 
\left. \frac{d \sqrt{\tau}}{t} \right|_{\tau=\tauu_1 t}^{\tau=\tauu_2 t} + \frac{1}{t} \int_{\tauu_1 t}^{\tauu_2 t} \tau^{3/2}
\Big( 2 \left| \Ric_{g(\tau)} \right|_{g(\tau)}^2 (\gamma(\tau)) 
+ \Delta_{g(\tau)} R_{g(\tau)}(\gamma(\tau)) \nonumber \\
& \hspace{12em} 
- \frac{2}{\tau} R_{g(\tau)}(\gamma(\tau)) 
- 2 \Ric_{g(\tau)}(\dot\gamma(\tau), \dot\gamma(\tau)) \Big) d \tau \label{hesslambda}
\end{align}
and consequently
\begin{align*}
\frac{\partial \Lambda}{\partial t}(t,x,y) 
+ & 
\sum_{i=1}^d \left. \Hess_{g(\tau_1) \oplus g(\tau_2)} \Lambda \right|_{(t, x, y)} (\xi_i, \xi_i)\\
& \leq 
\frac{d}{\sqrt{t}} \left( \sqrt{\tauu_2} - \sqrt{\tauu_1} \right)
- \frac{1}{2 t} \int_{\tauu_1 t}^{\tauu_2 t} \sqrt{\tau} \left( R_{g(\tau)} (\gamma(\tau)) + |\dot \gamma(\tau)|_{g(\tau)}^2 \right) d \tau\\
& = 
\frac{d}{\sqrt{t}} \left( \sqrt{\tauu_2} - \sqrt{\tauu_1} \right) - \frac{1}{2 t} \Lambda(t,x,y).
\end{align*}
\end{proposition}
The proof of Theorem~\ref{mainresult} is now achieved 
under Assumption~\ref{ass:empty} 
by combining Lemma~\ref{itoformula} and Proposition~\ref{lambda}:

\begin{proof}[Proof of Theorem~\ref{mainresult} under Assumption~\ref{ass:empty}]
Lemma~\ref{lem:i'bility} below ensures that 
$\Theta ( t, \tilde{X}_t , \tilde{Y}_t )$ is integrable. 
Thus it suffices to show that 
the bounded variation part of 
$\Theta (t , \tilde{X}_t , \tilde{Y}_t )$ 
is nonpositive. 
By Lemma~\ref{itoformula},
\begin{align*}
d \Theta (t, \tilde{X}_t, \tilde{Y}_t) 
& = 
\Bigg[ 
\frac{\partial \Theta}{\partial t}(t, \tilde{X}_t, \tilde{Y}_t) \\
& \hspace{1.5em} 
+ 
\sum_{i=1}^d 
\left. 
  \Hess_{g(\tauu_1 t) \oplus g(\tauu_2 t)} \Theta 
\right|_{(t,\tilde{X}_t,\tilde{Y}_t)}
\left( 
  \sqrt{\tauu_1} \tilde{U}_t e_i 
  \oplus 
  \sqrt{\tauu_2} \tilde{V}_t e_i^*
  , 
  \sqrt{\tauu_1} \tilde{U}_t e_i 
  \oplus 
  \sqrt{\tauu_2} \tilde{V}_t e_i^* 
\right) 
\Bigg] dt\\
& \hspace{1.5em} 
+ 
\sum_{i = 1}^d 
\left[ 
  \sqrt{2 \tauu_1} \tilde{U}_t e_i 
  \oplus 
  \sqrt{2 \tauu_2} \tilde{V}_t e_i^* 
\right] 
\Theta(t,\tilde{X}_t,\tilde{Y}_t) d W_t^i.
\end{align*}
For the bounded variation part we obtain
\[
\frac{\partial \Theta}{\partial t}(t, \tilde{X}_t, \tilde{Y}_t) 
= 
\frac{\sqrt{\tauu_2} - \sqrt{\tauu_1}}{\sqrt{t}} 
\Lambda(t, \tilde{X}_t, \tilde{Y}_t)
+ 
2 \left( \sqrt{\tauu_2 t} - \sqrt{\tauu_1 t} \right) 
\frac{\partial \Lambda}{\partial t} (t, \tilde{X}_t, \tilde{Y}_t) 
- 2d \left( \sqrt{\tauu_2} - \sqrt{\tauu_1} \right)^2
\]
and
\begin{multline*}
\sum_{i=1}^d 
\left. 
  \Hess_{g(\tauu_1 t) \oplus g(\tauu_2 t)} \Theta 
\right|_{(t,\tilde{X}_t,\tilde{Y}_t)}
\left( 
  \sqrt{\tauu_1} \tilde{U}_t e_i 
  \oplus 
  \sqrt{\tauu_2} \tilde{V}_t e_i^* 
  , 
  \sqrt{\tauu_1} \tilde{U}_t e_i 
  \oplus 
  \sqrt{\tauu_2} \tilde{V}_t e_i^* 
\right)\\
= 
2 \left( \sqrt{\tauu_2 t} - \sqrt{\tauu_1 t} \right) 
\sum_{i=1}^d 
\left. 
  \Hess_{g(\tauu_1 t) \oplus g(\tauu_2 t)} \Lambda 
\right|_{(t,\tilde{X}_t,\tilde{Y}_t)}
\left( 
  \sqrt{\tauu_1} \tilde{U}_t e_i 
  \oplus 
  \sqrt{\tauu_2} \tilde{V}_t e_i^* 
  , 
  \sqrt{\tauu_1} \tilde{U}_t e_i 
  \oplus 
  \sqrt{\tauu_2} \tilde{V}_t e_i^* 
\right).
\end{multline*}
Thus, by Proposition~\ref{lambda},
\begin{align*}
\frac{\partial \Theta}{\partial t} & (t, \tilde{X}_t, \tilde{Y}_t)
+ 
\sum_{i=1}^d 
\left. 
  \Hess_{g(\tauu_1 t) \oplus g(\tauu_2 t)} \Theta 
\right|_{(t,\tilde{X}_t,\tilde{Y}_t)}
\left( 
  \sqrt{\tauu_1} \tilde{U}_t e_i 
  \oplus 
  \sqrt{\tauu_2} \tilde{V}_t e_i^*
  , 
  \sqrt{\tauu_1} \tilde{U}_t e_i 
  \oplus 
  \sqrt{\tauu_2} \tilde{V}_t e_i^* 
\right)\\
& \leq 
2 \left( \sqrt{\tauu_2 t} - \sqrt{\tauu_1 t} \right) 
\left[ \frac{d}{\sqrt{t}} \left( \sqrt{\tauu_2} - \sqrt{\tauu_1} \right) 
- \frac{1}{2 t} \Lambda(t,\tilde{X}_t ,\tilde{Y}_t) \right] \\
& \hspace{12em} + 
\frac{\sqrt{\tauu_2} - \sqrt{\tauu_1}}{\sqrt{t}} 
\Lambda(t, \tilde{X}_t, \tilde{Y}_t) 
- 2d \left( \sqrt{\tauu_2} - \sqrt{\tauu_1} \right)^2\\
& = 0.
\end{align*}
Hence $\Theta(t,\tilde{X}_t,\tilde{Y}_t)$ is indeed a supermartingale.
\end{proof}

\begin{remark} \label{rem:mart}
Unlike the case in \cite{act}, 
the pathwise contraction of 
$\Theta ( t , \tilde{X}_t , \tilde{Y}_t )$ 
is no longer true in our case. 
In other words, 
the martingale part of 
$\Theta ( t , \tilde{X}_t , \tilde{Y}_t )$ 
does not vanish. 
We will see it in the following. 
The minimal $\LL$-geodesic 
$\gm = \gm_{xy}^{\tau_1 \tau_2}$ of $L ( x, \tau_1 ; y , \tau_2 )$ 
satisfies the $\LL$-geodesic equation 
\begin{equation} \label{eq:dt1}
\nabla_{\dot{\gm} (\tau)}^{g(\tau)} \dot{\gm} (\tau) 
= 
\frac12 \nabla^{g(\tau)} R_{g(\tau)} 
- 
2 \Ric_{g(\tau)}^\# ( \dot{\gm} (\tau) ) 
- 
\frac{1}{2\tau} \dot{\gm} (\tau) 
\end{equation}
(see \cite[Corollary~7.19]{chowetal}). 
Thus the first variation formula 
(see \cite[Lemma~7.15]{chowetal}) yields 
\begin{equation} \label{eq:first}
\sqrt{2 \tauu_1} \tilde{U}_t e_i 
\oplus 
\sqrt{2 \tauu_2} \tilde{V}_t e_i^* 
\Lambda (t,\tilde{X}_t,\tilde{Y}_t) 
= 
\sqrt{2t} \tauu_2 
\langle 
  \tilde{V}_t e_i^* , \dot{\gm} (\tauu_2 t)
\rangle_{g (\tauu_2 t)}
- 
\sqrt{2t} \tauu_1 
\langle 
  \tilde{U}_t e_i , \dot{\gm} (\tauu_1 t)
\rangle_{g(\tauu_1 t)} .
\end{equation}
One obstruction to pathwise contraction is 
on the difference of time-scalings $\tauu_1$ and $\tauu_2$. 
In addition, by \eqref{eq:dt1}, 
$\sqrt{\tau} \dot{\gm} (\tau)$ is \emph{not} 
space-time parallel to $\gm$ 
in general (cf. Remark~\ref{rem:isom}). 
\end{remark}

\section{Proof of Proposition~\ref{lambda}}
In this section,
we write $\tau_1 := \tauu_1 t$ and $\tau_2 := \tauu_2 t$.
We assume $\tau_2 < T$. 
For simplicity of notations, 
we abbreviate the dependency on the metric $g(\tau)$ of 
several geometric quantities such as 
$\Ric$, $R$, 
the inner product $\langle \cdot, \cdot \rangle$, 
the covariant derivative $\nabla$ 
etc.~when 
our choice of $\tau$ is obvious. 
For this abbreviation, 
we will think that 
$\gm (\tau)$ is in $( M , g (\tau) )$ 
and 
$\dot{\gm} ( \tau )$ is 
in $( T_{\gm (\tau)} M , g (\tau) )$. 
Note that, when $\tauu_1 = 0$, 
$\lim_{\tau \downarrow \tauu_1} \sqrt{\tau} \dot{\gm} (\tau)$ exists
while
$\lim_{\tau \downarrow 0} | \dot{\gm} (\tau) | = \infty$.
In any case, $\sqrt{\tau} | \dot{\gm} (\tau) |$ is bounded
(see \eqref{eq:velo-bound1}).

We first compute the time derivative of $\Lambda$. 
When $\tauu_1 > 0$,
by \cite[Formulas~(A.4) and (A.5)]{topping} we have
\begin{align*}
\frac{\partial L}{\partial \tau_1}(x, \tau_1; y, \tau_2) 
& = 
-\sqrt{\tau_1} \left( R_{g(\tau_1)}(x) - |\dot\gamma(\tau_1)|^2 \right),\\
\frac{\partial L}{\partial \tau_2}(x, \tau_1; y, \tau_2) 
& = 
\sqrt{\tau_2} \left( R_{g(\tau_2)}(y) - |\dot\gamma(\tau_2)|^2 \right),
\end{align*}
so that
\begin{align} 
\frac{\partial \Lambda}{\partial t}(t,x,y) 
& = 
\tauu_1 \frac{\partial L}{\partial \tau_1}(x, \tau_1; y, \tau_2) 
+ 
\tauu_2 \frac{\partial L}{\partial \tau_2}(x, \tau_1; y, \tau_2) 
\nonumber \\
& = 
\frac{1}{t} 
\left( 
  \tau_2^{3/2} 
  \left( 
    R (\gamma(\tau_2)) 
    - 
    |\dot\gamma(\tau_2)|^2 
  \right)
  - \tau_1^{3/2} 
  \left( 
    R (\gamma(\tau_1)) 
    - 
    |\dot\gamma(\tau_1)|^2 
  \right) 
\right) 
. 
\label{eq:dt}
\end{align}
Thus the integration-by-parts yields,
\begin{align}
\frac{\partial \Lambda}{\partial t}(t,x,y) 
& = 
\frac{3}{2 t} 
\int_{\tau_1}^{\tau_2} \sqrt{\tau} 
\left( 
  R (\gamma(\tau)) - |\dot\gamma(\tau)|^2 
\right) d \tau \nonumber \\
& \quad 
+ \frac{1}{t} 
\int_{\tau_1}^{\tau_2} 
\tau^{3/2} \bigg( 
  \frac{\partial R}{\partial \tau}(\gamma(\tau)) 
  + 
  \nabla_{\dot\gamma(\tau)} R (\gamma(\tau)) \nonumber \\
& \hspace{8em} 
  - 2 
  \langle 
    \nabla_{\dot{\gamma}(\tau)} \dot\gamma(\tau) , \dot\gamma(\tau) 
  \rangle
  - 2 \Ric (\dot\gamma(\tau), \dot\gamma(\tau)) 
\bigg) d \tau. \label{eq:dt0}
\end{align}
Note that we have 
\begin{equation} \label{eq:dt2} 
\frac{\partial R}{\partial \tau} 
= 
- \Delta R - 2 \, |\!\Ric |^2
\end{equation}
(see e.g.~\cite[Proposition~2.5.4]{toppinglectures}).
Since $\gm$ satisfies the $\LL$-geodesic equation \eqref{eq:dt1}, 
by substituting \eqref{eq:dt1} and \eqref{eq:dt2} 
into \eqref{eq:dt0}, we obtain \eqref{dlambdadt}. 
Note that
the derivation of \eqref{eq:dt} and \eqref{eq:dt0} is 
still valid even when $\tauu_1 = 0$ 
because of the remark at the beginning of this section. 
Thus \eqref{dlambdadt} holds when $\tauu_1 = 0$, too.

In order to estimate $\sum_{i=1}^d \left. \Hess_{g(\tau_1) \oplus g(\tau_2)} \Lambda \right|_{(t, x, y)} (\xi_i, \xi_i)$
we begin with the second variation formula for the $\LL$-functional:

\begin{lemma}[Second variation formula; {\cite[Lemma~7.37]{chowetal}}]\label{secondvariation}
Let $\Gamma: (-\eps, \eps) \times [\tau_1, \tau_2] \to M$ be a variation of $\gamma$, $S(s, \tau) := \partial_s \Gamma(s, \tau)$,
and $Z(\tau) := \partial_s \Gamma(0,\tau)$ the variation field of $\Gamma$. Then
\begin{align}
\left. \frac{d^2}{ds^2} \right|_{s=0} \LL(\Gamma_s) 
& = 
2 \sqrt{\tau} 
\left. 
  \langle 
    \dot\gamma(\tau) 
    , 
    \nabla_{Z(\tau)} S(0, \tau) 
  \rangle 
\right|_{\tau=\tau_1}^{\tau=\tau_2}
- 2 
\left. 
  \sqrt{\tau} 
  \Ric (Z(\tau), Z(\tau)) 
\right|_{\tau=\tau_1}^{\tau=\tau_2} \nonumber \\
& \qquad 
+ 
\left. 
  \frac{1}{\sqrt{\tau}} 
  \left| Z(\tau) \right|^2 
\right|_{\tau=\tau_1}^{\tau=\tau_2}
- 
\int_{\tau_1}^{\tau_2} 
  \sqrt{\tau} 
  H(\dot\gamma(\tau), Z(\tau)) 
d \tau \nonumber \\
& \qquad 
+ 
\int_{\tau_1}^{\tau_2} 
  2 \sqrt{\tau} 
  \left| 
    \nabla_{\dot{\gm} (\tau)} Z(\tau) 
    + 
    \Ric^\#(Z(\tau)) 
    - 
    \frac{1}{2 \tau} Z(\tau) 
  \right|^2 
d \tau, \label{zweite}
\end{align}
where
\begin{align} \nonumber
H(\dot\gamma(\tau), Z(\tau)) 
& := 
-2 \frac{\partial \Ric}{\partial \tau}(Z(\tau), Z(\tau)) 
- \Hess R (Z(\tau), Z(\tau))
+ 2 \, | \! \Ric^\# (Z(\tau))|^2
\\ \nonumber
& \qquad 
- \frac{1}{\tau} \Ric (Z(\tau), Z(\tau)) 
- 2 \Rm (Z(\tau), \dot\gamma(\tau), \dot\gamma(\tau), Z(\tau))
\\ \label{eq:H}
& \qquad 
- 4 (\nabla_{\dot\gamma(\tau)} \Ric)(Z(\tau), Z(\tau)) 
+ 4 (\nabla_{Z(\tau)} \Ric)(\dot\gamma(\tau), Z(\tau)).
\end{align}
\end{lemma}

In \cite{chowetal} this lemma is only proved in the case $\tau_1 = 0$ and $Z(\tau_1) = 0$.
However, the proof given there can be easily adapted to the slightly more general case needed here.

\begin{corollary}[see {\cite[Lemma~7.39]{chowetal}} for a similar statement] \label{secondvariationzwei}
If the variation field $Z$ is of the form 
\begin{equation}\label{v}
Z(\tau) = \sqrt{\frac{\tau}{t}} Z^* (\tau)
\end{equation}
with a space-time parallel field $Z^*$ 
satisfying $|Z^*(\tau)| \equiv 1$, 
then
\begin{align*}
\left. \frac{d^2}{ds^2} \right|_{s=0} \LL(\Gamma_s)
& = 
2 \sqrt{\tau} 
\left. 
  \langle 
    \dot\gamma(\tau) 
    ,  
    \nabla_{Z(\tau)} S(0, \tau) 
  \rangle_{g(\tau)} 
\right|_{\tau=\tau_1}^{\tau=\tau_2}
- 2 
\left. 
  \sqrt{\tau} \Ric (Z(\tau), Z(\tau)) 
\right|_{\tau=\tau_1}^{\tau=\tau_2}
\\
& \quad 
- 
\int_{\tau_1}^{\tau_2} 
  \sqrt{\tau} H(\dot\gamma(\tau), Z(\tau)) 
d \tau 
+ 
\left. 
  \frac{\sqrt{\tau}}{t} 
\right|_{\tau=\tau_1}^{\tau=\tau_2} .
\end{align*}
\end{corollary}

\begin{proof}
Since $Z^*$ is space-time parallel, $Z$ satisfies
\begin{equation}\label{dgl}
\nabla_{\dot{\gm}(\tau)} Z(\tau) 
= 
-\Ric^\# ( Z(\tau) ) 
+ \frac{1}{2 \tau} Z(\tau), 
\end{equation}
so that the last term in \eqref{zweite} vanishes.
\end{proof}

\begin{corollary}[Hessian of $L$; see {\cite[Corollary~7.40]{chowetal}} for a similar statement]\label{hessianofl}
Let $Z$ be a vector field along $\gamma$ of the form \eqref{v} 
and 
$\xi := Z(\tau_1) \oplus Z(\tau_2) \in T_{(x,y)} (M \! \times \! M)$
Then
\begin{align} 
\left. \Hess_{g(\tau_1) \oplus g(\tau_2)} L \right|_{(x, \tau_1; y, \tau_2)} (\xi, \xi)
& \leq 
- 
\int_{\tau_1}^{\tau_2} 
  \sqrt{\tau} H(\dot\gamma(\tau), Z(\tau)) 
d \tau 
+ 
\left. 
  \frac{\sqrt{\tau}}{t} 
\right|_{\tau=\tau_1}^{\tau=\tau_2} 
\nonumber \\
& \quad 
- 2 
\left. 
  \sqrt{\tau} \Ric_{g(\tau)} (Z(\tau), Z(\tau)) 
\right|_{\tau=\tau_1}^{\tau=\tau_2}. 
\label{hessl}
\end{align}
\end{corollary}

\begin{proof}
Let $\Gamma: (-\eps, \eps) \times [\tau_1, \tau_2] \to M$ be any variation of $\gamma$ with variation field $Z$ and such that 
$\nabla_{Z(\tau_1)} S(0, \tau_1)$ 
and 
$\nabla_{Z(\tau_2)} S(0, \tau_2)$ vanish. 
Since
\[
\left. \Hess_{g(\tau_1) \oplus g(\tau_2)} L \right|_{(x, \tau_1; y, \tau_2)} (\xi, \xi) \leq \left. \frac{d^2}{ds^2} \right|_{s=0} \LL (\Gamma_s),
\]
the claim follows from Corollary~\ref{secondvariationzwei}.
\end{proof}

Let now $Z_i^*$ ($i = 1, \ldots, d$) be 
space-time parallel fields along $\gamma$ 
satisfying $Z_i^* (\tau_1) = u e_i$ 
(and consequently $Z_i^* (\tau_2) = v e_i^*$),
and $Z_i(\tau) := \sqrt{ \tau / t } Z_i^*(\tau)$ 
(so that $\xi_i = Z_i(\tau_1) \oplus Z_i(\tau_2)$). 
In order to estimate
$\sum_{i=1}^d \left. \Hess_{g(\tau_1) \oplus g(\tau_2)} L \right|_{(x, \tau_1; y, \tau_2)} (\xi_i, \xi_i)$
using Corollary~\ref{hessianofl} 
we will compute $\sum_{i=1}^d H(\dot\gamma(\tau), Z_i(\tau))$ 
in the following 
(see \cite[Section~7.5.3]{chowetal} for a similar argument).
Set $I_1$, $I_2$ and $I_3$ by 
\begin{align*}
I_1 
& := 
-2 
\sum_{i=1}^d
\frac{\partial \Ric}{\partial \tau}(Z_i(\tau), Z_i(\tau)) 
,
\\
I_2 
& := 
\sum_{i=1}^d 
\Big[
  - \Hess R (Z_i(\tau), Z_i(\tau)) 
  + 2 \, | \! \Ric^\# (Z_i(\tau))|^2 
\\
& \hspace{8em}
  - \frac{1}{\tau} \Ric (Z_i(\tau), Z_i(\tau))
  - 2 \Rm (Z_i(\tau), \dot\gamma(\tau), \dot\gamma(\tau), Z_i(\tau))
\Big]
,
\\
I_3 
& := 
4 \sum_{i=1}^d
\Big[ 
  (\nabla_{Z_i(\tau)} \Ric)( Z_i(\tau), \dot\gamma(\tau)) 
  - 
  (\nabla_{\dot\gamma(\tau)} \Ric)(Z_i(\tau), Z_i(\tau)) 
\Big].
\end{align*}
Then 
$\sum_{i=1}^d H(\dot\gamma(\tau), Z_i(\tau)) = I_1 + I_2 + I_3$ holds.   
By a direct computation, 
\begin{equation} \label{eq:I3}
I_2 = 
\frac{\tau}{t} 
\left( 
    - \Delta R (\gamma(\tau)) 
    + 2 \left| \Ric \right|^2 (\gamma(\tau)) 
    - \frac{1}{\tau} R (\gamma(\tau))
    + 2 \Ric(\dot\gamma(\tau), \dot\gamma(\tau)) 
\right) .
\end{equation}
The contracted Bianchi identity 
$\dive \Ric = \frac{1}{2} \nabla R$ \cite[Lemma~7.7]{lee}
yields 
\begin{equation} \label{eq:I2}
I_3 
= 
\frac{4 \tau}{t} 
\left( 
    ( \dive \Ric )(\dot\gamma(\tau))
    - 
    (\nabla_{\dot\gamma(\tau)} R)(\gamma(\tau))
\right)
= 
- \frac{2\tau}{t} 
(\nabla_{\dot\gamma(\tau)} R)(\gamma(\tau))
. 
\end{equation}
For $I_1$, we have 
\begin{align} \nonumber 
I_1 
& = 
-2 \sum_{i=1}^d 
\left[ 
    \frac{d}{d \tau}( \Ric (Z_i(\tau), Z_i(\tau)) ) 
    - (\nabla_{\dot\gamma(\tau)} \Ric) (Z_i(\tau), Z_i(\tau)) 
    - 2 \Ric( \nabla_{\dot{\gm}(\tau)} Z_i(\tau), Z_i(\tau)) 
\right]
\\ \nonumber
& = 
- 2 \frac{d}{d \tau} 
\left( 
    \frac{\tau}{t} R (\gamma(\tau)) 
\right)
+ 2 \frac{\tau}{t} \nabla_{\dot\gamma(\tau)} R ( \gm(\tau) )
+ 4 \sum_{i=1}^d \Ric( \nabla_{\dot{\gm}(\tau)} Z_i(\tau), Z_i(\tau)) 
\\ \label{eq:I1}
& = 
-\frac{2\tau}{t} 
\left( 
    \frac{1}{\tau} R (\gamma(\tau)) 
    + \frac{\partial R}{\partial \tau}(\gamma(\tau))
\right)
+ 4 
\sum_{i=1}^d \Ric( \nabla_{\dot{\gm}(\tau)} Z_i(\tau), Z_i(\tau)) 
.
\end{align}
Since $Z_i$ satisfies \eqref{dgl},
\begin{align} \nonumber
4 \sum_{i=1}^d \Ric( \nabla_{\dot{\gm}(\tau)} Z_i(\tau), Z_i(\tau)) 
& = 
4 \sum_{i=1}^d 
\Ric( 
  -\Ric^\# ( Z_i(\tau) ) + \frac{1}{2 \tau} Z_i(\tau) 
  , 
  Z_i(\tau)
)
\\ \label{eq:I11}
& = 
- \frac{2\tau}{t} 
\left( 
    2 \left| \Ric \right|^2 (\gamma(\tau)) 
    - 
    \frac{1}{\tau} R (\gamma(\tau)) 
\right) 
.
\end{align}
By substituting \eqref{eq:I11} into \eqref{eq:I1}, 
\begin{equation} \label{eq:I12}
I_1 
= 
- \frac{2\tau}{t} 
\left( 
    \frac{\partial R}{\partial \tau}(\gamma(\tau))
    + 
    2 \left| \Ric \right|^2 (\gamma(\tau)) 
\right). 
\end{equation}
Hence, 
by combining \eqref{eq:I12}, \eqref{eq:I2} and \eqref{eq:I3}, 
\begin{align*}
\sum_{i=1}^d H(\dot\gamma(\tau), Z_i(\tau))
& = 
\frac{\tau}{t} 
\Big( 
- 2 \frac{\partial R}{\partial \tau}(\gamma(\tau)) 
- 2 \left| \Ric \right|^2 (\gamma(\tau))
- \Delta R (\gamma(\tau)) 
\\ 
& \hspace{6em} 
- \frac{1}{\tau} R (\gamma(\tau))
+ 2 \Ric (\dot\gamma(\tau), \dot\gamma(\tau)) 
- 2 (\nabla_{\dot\gamma(\tau)} R)(\gamma(\tau)) 
\Big)
.
\end{align*}
Inserting this into \eqref{hessl} we obtain
\begin{align*}
\sum_{i=1}^d & 
\left. \Hess_{g(\tau_1) \oplus g(\tau_2)} L \right|_{(x, \tau_1; y, \tau_2)} (\xi_i, \xi_i)
\\
& \leq 
\frac{1}{t} 
\int_{\tau_1}^{\tau_2} \tau^{3/2}
\Big( 
  2 \frac{\partial R}{\partial \tau}(\gamma(\tau)) 
  + 
  2 \left| \Ric \right|^2 (\gamma(\tau))
  + 
  \Delta R (\gamma(\tau)) 
\\
& \hspace{8em} 
  + 
  \frac{1}{\tau} R (\gamma(\tau))
  - 
  2 \Ric (\dot\gamma(\tau), \dot\gamma(\tau)) 
  + 
  2 (\nabla_{\dot\gamma(\tau)} R)(\gamma(\tau)) 
\Big) d \tau
\\
& \qquad 
+ 
\left. 
    \frac{d \sqrt{\tau}}{t} 
\right|_{\tau=\tau_1}^{\tau=\tau_2} 
- 
\left. 
    \frac{2 \tau^{3/2}}{t} R (\gamma(\tau)) 
\right|_{\tau=\tau_1}^{\tau=\tau_2}\\
& = 
\left. 
    \frac{d \sqrt{\tau}}{t} 
\right|_{\tau=\tau_1}^{\tau=\tau_2} 
+ 
\frac{1}{t} \int_{\tau_1}^{\tau_2} \tau^{3/2}
\Big( 
  2 \left| \Ric \right|^2 (\gamma(\tau)) 
  + 
  \Delta R (\gamma(\tau)) 
  - \frac{2}{\tau} R (\gamma(\tau))
  - 2 \Ric (\dot\gamma(\tau), \dot\gamma(\tau) ) 
\Big) d \tau 
\end{align*}
which completes the proof of Proposition~\ref{lambda}.

\section{Coupling via approximation by geodesic random walks}
\label{randomwalks}

To avoid a technical difficulty coming from 
singularity of $L$ on the $\LL$-cut locus, 
we provide an alternative way to constructing 
a coupling of Brownian motions 
by space-time parallel transport. 
In this section, we first define 
a coupling of geometric random walks 
which approximate $g(\tau)$-Brownian motion.  
Next, in order to provide a local uniform 
control of error terms 
coming from our discretization, 
we study several estimates 
of geometric quantities 
in subsection~\ref{sec:geom}. 
Those are obtained 
as a small modification of existing arguments 
in \cite{chowetal,topping,ye}. 
The $\LL$-cut locus is also reviewed and studied there. 
Finally, we will establish an analogue of arguments 
in section~\ref{sec:SDE} for coupled geodesic random walks 
to complete the proof of Theorem~\ref{mainresult}. 

Let us take a family of minimal $\LL$-geodesics 
$
\{  
 \gm_{xy}^{\tau_1 \tau_2} 
\; | \;
   \tauu_1 
    \le \tau_1 
    < \tau_2 
    \le \tauu_2,
   x, y \in M
\}
$ 
so that a map 
$
( x , \tau_1 ; y , \tau_2 ) 
\mapsto 
\gm_{xy}^{\tau_1 \tau_2}
$ 
is measurable. 
The existence of such a family of minimal $\LL$-geodesics 
can be shown in a similar way 
as discussed in the proof of 
\cite[Proposition~2.6]{lottvillani} 
since the family of 
minimal $\LL$-geodesics with fixed endpoints 
is compact 
(cf.~\cite[the proof of Lemma~7.27]{chowetal}). 
For each $\tau \in [ \tauu_1 , T ]$, 
take a measurable section $\Phi^{(\tau)}$ 
of $g (\tau)$-orthonormal frame bundle $\OO^{g(\tau)} (M)$ of $M$. 
For $x,y \in M$ and 
$\tau_1 , \tau_2 \in [ \tauu_1 , T ]$ with $\tau_1 < \tau_2$, 
let us define 
$
\Phi_i ( x , \tau_1 ; y , \tau_2 )
\in \mathcal{F} ( M )
$ 
for $i=1,2$ by 
\begin{align*} 
\Phi_1 ( x, \tau_1 ; y, \tau_2 ) 
& := 
\Phi^{(\tau_1)} (x) ,
\\
\Phi_2 ( x, \tau_1 ; y , \tau_2 ) 
& := 
m_{xy}^{\tau_1 \tau_2} \circ \Phi^{(\tau_1)} (x)
,
\end{align*}
where $m_{xy}^{\tau_1 \tau_2}$ is 
as given in Definition~\ref{def:pt}. 
Let us take a family of $\R^d$-valued 
i.i.d.~random variables $( \lambda_n )_{n \in \N}$ 
which are uniformly distributed 
on a unit ball centered at origin. 
We denote the (Riemannian) exponential map 
with respect to $g(\tau)$ at $x \in M$ 
by $\exp^{(\tau)}_x$. 
In what follows, 
we define a coupled geodesic random walk 
$
\mathbf{X}^\eps_t 
 = 
( X^\eps_{ \tauu_1 t } , Y^\eps_{ \tauu_2 t } )
$ 
with scale parameter $\eps > 0$ 
and 
initial condition 
$\mathbf{X}^\eps_s = ( x_1 , y_1 )$ 
inductively. 
First we set 
$( X^\eps_{ \tauu_1 s} , Y^\eps_{ \tauu_2 s} ) := ( x_1 , y_1 )$. 
For simplicity of notations, 
we set $t_n := ( s + \eps^2 n ) \wedge (T / \tauu_2)$. 
After we defined $( \mathbf{X}^\eps_t )_{t \in [ s , t_n ]}$, 
we extend it to $( \mathbf{X}^\eps_t )_{t \in [ s,  t_{n+1} ]}$ 
by 
\begin{align*}
\hat{\lambda}_{n+1}^{(i)} 
& : = 
\sqrt{d+2} 
\Phi_i ( 
 X^\eps_{\tauu_1 t_n} , 
 \tauu_1 t_n ; 
 Y^\eps_{\tauu_2 t_n} , 
 \tauu_2 t_n  
) 
\lambda_{n+1} 
,
\qquad i = 1, 2 ,
\\   
X^\eps_{ \tauu_1 t } 
& := 
\exp_{X^\eps_{ \tauu_1 t_n }}^{( \tauu_1 t_n )} 
\abra{
  \frac{ t - t_n }{\eps} 
  \sqrt{2 \tauu_1} 
  \hat{\lambda}_{n+1}^{(1)}
} , 
\\ 
Y^\eps_{ \tauu_2 t } 
& := 
\exp_{Y^\eps_{ \tauu_2 t_n } }^{( \tauu_2 t_n )} 
\abra{ 
  \frac{ t - t_n }{\eps} 
  \sqrt{2 \tauu_2} 
  \hat{\lambda}_{n+1}^{(2)} 
}. 
\end{align*}
We can (and we will) extend the definition of $X_{\tau}^\eps$ 
for $\tau \in [ T \tauu_1 / \tauu_2 , T ]$ in the same way. 
As in section \ref{sec:SDE},
$X_{\tauu_1 t}^\eps$ does not move when $\tauu_1 = 0$.
Note that $\sqrt{d+2}$ is a normalization factor 
in the sense 
$
\mathop{\mathrm{Cov}} ( \sqrt{d+2} \lambda_n ) 
= 
\mathrm{Id}
$.   
Let us equip path spaces $C ( [ a , b ] \to M )$ 
or $C ( [ a, b ] \to M \times M )$ 
with the uniform convergence topology
induced from $g(T)$. 
Here the interval $[a,b]$ will be chosen
appropriately in each context. 
By \eqref{eq:metric-bound} which we will review below,
different choices of a metric $g(\tau)$ from $g(T)$ 
always induce the same topology on path spaces. 
As shown in \cite{kuwada2}, 
$( X^\eps_\tau )_{\tau \in [ \tauu_1 s , T ]}$ and 
$( Y^\eps_\tau )_{\tau \in [\tauu_2 s , T ]}$ 
converge in law to 
$g(\tau)$-Brownian motions 
$( X_\tau )_{\tau \in [ \tauu_1 s , T ]}$ 
and 
$( Y_\tau )_{\tau \in [ \tauu_2 s , T ]}$ 
on $M$ 
with initial conditions 
$X_{\tauu_1 s} = x_1$, $Y_{\tauu_2 s} = y_1$ 
respectively 
as $\eps \to 0$ (when $\tauu_1 > 0$).
As a result, $\mathbf{X}^\eps$ is tight and 
hence there is a convergent subsequence of 
$\mathbf{X}^\eps$. 
We fix such a subsequence and use the same symbol 
$( \mathbf{X}^\eps )_\eps$ for simplicity of notations. 
We denote the limit in law of $\mathbf{X}^\eps$ 
as $\eps \to 0$ 
by 
$
\mathbf{X}_t
= 
( 
 X_{ \tauu_1 t } , 
 Y_{ \tauu_2 t } 
)
$. 
Recall that, in this paper, 
$g( \tau )$-Brownian motion 
means a time-inhomogeneous diffusion process 
associated with $\Delta_{g(\tau)}$ 
instead of $\Delta_{g(\tau)} / 2$. 

\begin{remark} \label{rem:gRW}
We explain the reason 
why our alternative construction works efficiently 
to avoid the problem arising from singularity of $L$. 
To make it clear, 
we begin with observing the essence of difficulties 
in the SDE approach we used in section~\ref{sec:SDE}. 
Recall that our argument is based on the It\^{o} formula. 
Hence non-differentiability of $L$ 
at the $\LL$-cut locus 
causes the technical difficulty. 
One possible strategy is to extend 
the It\^{o} formula for $\LL$-distance. 
Since $\LL$-cut locus is sufficiently thin, 
we can expect that 
the totality of times 
when our coupled particles stay there 
has measure zero. 
In addition, as that of Riemannian cut locus, 
the presence of $\LL$-cut locus would work 
to decrease the $\LL$-distance between coupled particles. 
Thus one might think it possible to extend It\^{o} formula 
for $\LL$-distance to the one involving 
a ``local time at the $\LL$-cut locus''. 
If we succeed in doing so, 
we will obtain a differential inequality 
which implies the supermartingale property 
by neglecting this additional term 
since it should be nonpositive. 

Instead of completing the above strategy, 
our alternative approach in this section 
directly provides a difference inequality 
without extracting the additional ``local time'' term. 
By dividing a minimal $\LL$-geodesic into two pieces, 
we can obtain a ``difference inequality'' of 
$\LL$-distance even when 
the pair of endpoints belongs to the $\LL$-cut locus 
(see Lemma~\ref{lem:2var-gRW}). 
In order to employ such an inequality, 
it is more suitable to work with discrete time processes. 
\end{remark}

\subsection{Preliminaries on the geometry of $\LL$-functional} 
\label{sec:geom}

Recall that we assume that $M$ has bounded curvature, 
so that there is a constant $C_0 < \infty $ 
such that 
%
\begin{equation} \label{eq:curv-bound}
\max_{ 
  ( x, \tau ) \in M \times [ \tauu_1 , T ] 
} 
| \Rm |_{g(\tau)} (x)
\vee 
| \Ric |_{g(\tau)} (x)
\le C_0.
\end{equation}
%
On the basis of \eqref{eq:curv-bound}, 
we have a comparison of Riemannian metrics at different times. 
That is, for $\tau_1 < \tau_2$,
\begin{equation} \label{eq:metric-bound}
\e^{-2 C_0 ( \tau_2 - \tau_1 )} 
g ( \tau_2 )
 \le 
g ( \tau_1 )
 \le 
\e^{2 C_0 ( \tau_2 - \tau_1 )}
g ( \tau_2 )
.
\end{equation}
Let $\rho_{g(\tau)}$ be 
the distance function on $M$ 
at time $\tau$. 
Note that a similar comparison between 
$\rho_{g (\tau_1)}$ and $\rho_{g(\tau_2)}$ 
follows from \eqref{eq:metric-bound}. 
We also obtain 
the following bounds for $L$ 
from \eqref{eq:curv-bound} and \eqref{eq:metric-bound}. 
Let $\gm \: : \: [ \tau_1 , \tau_2 ] \to M$ 
be a minimal $\LL$-geodesic. 
Then, for $\tau \in [ \tau_1 , \tau_2 ]$, 
\begin{multline} \label{eq:L-bound} 
\frac
    { \e^{-2 C_0 ( \tau_2 - \tau_1 ) } }
    {2 \sqrt{\tau_2} - \sqrt{\tau_1} } 
\rho_{g(T)} ( \gm (\tau_1) , \gm (\tau) )^2  
- 
\frac23 
d C_0 ( \tau_2^{3/2} - \tau_1^{3/2} ) 
\le 
L ( \gm (\tau_1) , \tau_1 ; \gm (\tau_2) , \tau_2 ) 
\\
\le 
\frac
    { \e^{2 C_0 ( \tau_2 - \tau_1 ) } } 
    {2 \sqrt{\tau_2} - \sqrt{\tau_1} } 
\rho_{g(T)} ( \gm ( \tau_1 ) , \gm ( \tau_2 ) )^2 
+  
\frac23 
d C_0 ( \tau_2^{3/2} - \tau_1^{3/2} ) 
\end{multline}
(see \cite[Lemma~7.13]{chowetal} 
and 
\cite[Proposition~B.2]{topping}). 
The same estimate holds for 
$\rho_{g(T)} ( \gm (\tau) , \gm (\tau_2) )^2$ 
instead of 
$\rho_{g(T)} ( \gm (\tau_1) , \gm (\tau) )^2$. 
Taking the fact that 
$\LL$-functional is \emph{not} invariant under 
re-parametrization of curves into account, 
we will introduce an estimate for the velocity of 
the minimal $\LL$-geodesic $\gm$. 
By a similar argument 
as in \cite[Lemma~7.13 (ii)]{chowetal}, 
there exists $\tau^* \in [ \tau_1 , \tau_2 ]$ 
such that 
\begin{equation} \label{eq:velo-bound0}
\tau^* 
\left| 
  \dot{\gm} (\tau^*) 
\right|_{g(\tau^*)}^2 
\le 
\frac{1}{2 ( \sqrt{\tau_2} - \sqrt{\tau_1} )}
\abra{ 
  L ( \gm (\tau_1) , \tau_1 ; \gm (\tau_2) , \tau_2 )
   + 
  \frac{ 2 d C_0 }{ 3 } ( \tau_2^{3/2} - \tau_1^{3/2} ) 
}
. 
\end{equation} 
Suppose $\tau_2 < T$. 
Then, as shown in \cite{shi} 
(see \cite{chowknopf} also), 
\eqref{eq:curv-bound} yields that 
there is a constant $C (d) >0$ 
depending only on $d$ 
such that 
\begin{equation} \label{eq:curv-bound1} 
\sup_{
  \tau \le \tau_2 , \, x \in M
} 
\left|
  \nabla \Rm  
\right|_{g(\tau)} (x) 
\le 
\frac{ C (d) C_0 }{ ( T - \tau_2 ) \wedge C_0^{-1} } =: C_0'
. 
\end{equation}
By virtue of \eqref{eq:curv-bound1}, 
there exists a constant $c_1, C_1 > 0$ 
which depends on 
$C_0$, $C_0'$ and $T$ such that 
for all $\tau_1' , \tau_2' \in [ \tau_1 , \tau_2 ]$ 
with $\tau_1' < \tau_2 '$, 
\begin{align} \label{eq:v-u-bound}
\tau_2' 
\left| 
 \dot{\gm} (\tau_2')  
\right|_{g (\tau_2')}^2  
& \le 
c_1 
\tau_1' 
\left| 
 \dot{\gm} (\tau_1') 
\right|_{g (\tau_1')}^2  
+ C_1 
, 
\\ \label{eq:v-l-bound}
\tau_1' 
\left| 
 \dot{\gm} (\tau_1')  
\right|_{g (\tau_1')}^2  
& \le 
c_1 
\tau_2' 
\left| 
 \dot{\gm} (\tau_2')  
\right|_{g (\tau_2')}^2 
+ C_1 
.
\end{align}
The first inequality in \eqref{eq:v-u-bound} 
can be shown similarly as \cite[Lemma~7.24]{chowetal}.   
It is due to a differential inequality 
based on the $\LL$-geodesic equation \eqref{eq:dt1}
which provides an upper bound of 
$
\partial_\tau 
( \tau | \dot{\gm} (\tau) |_{g(\tau)}^2 )
$. 
By considering 
a lower bound of the same quantity instead, 
we obtain the second inequality 
\eqref{eq:v-l-bound} 
in a similar way. 
Combining \eqref{eq:v-u-bound} and \eqref{eq:v-l-bound} 
with \eqref{eq:velo-bound0} and \eqref{eq:L-bound}, 
we can take constants $c_2 > 0$ and $C_2 > 0$ 
depending on $C_0$, $c_1$, $C_1$, $\tau_1$ and $\tau_2$ 
such that 
\begin{equation} \label{eq:velo-bound1}
\tau
\left| 
 \dot{\gm} (\tau)
\right|_{g(\tau)}^2 
\le 
c_2 
\rho_{g(T)} ( \gm (\tau_1) , \gm(\tau_2) )^2  
 + 
C_2 
\end{equation} 
for $\tau_1 \le \tau \le \tau_2$. 
Though $c_2$ and $C_2$ depends on $\tau_1$ and $\tau_2$, 
it is easy to see that 
$c_2$ and $C_2$ are uniformly bounded above 
as long as 
$\tau_2 - \tau_1$ and $T - \tau_2$ 
is uniformly away from 0. 

Let us recall the definition and some properties 
of $\LL$-cut locus according to 
\cite{chowetal,topping,ye}. 
Given $\tau , \tau' \in [ 0 , T )$ 
with $\tau < \tau'$ 
and $x \in M$, 
we define $\LL$-exponential map 
$\LL_{\tau, \tau'} \exp_{x} : T_x M \to M$ 
by $\LL_{\tau, \tau'} \exp_{x} (Z) = \gm (\tau')$, 
where $\gm$ is a unique $\LL$-geodesic 
from $(\tau , x)$ 
with the initial condition 
$\lim_{\tau' \downarrow \tau} \sqrt{\tau'} \dot{\gm} (\tau') = Z$.
Note that we can extend the domain of 
$\LL$-geodesic $\gm$ to the interval $[ \tau, T )$ 
by using \eqref{eq:v-u-bound} 
(see \cite[Lemma~7.25]{chowetal}). 
Set 
\begin{align*}
\Omega ( x ,\tau_1 ; \tau_2 )
& : = 
\bbra{ 
  Z \in T_x M
    \; \left| \; 
  \begin{array}{l}
    \gm \: : \: [ \tau_1 , \tau_2 ] \to M 
    \mbox{ 
      given by 
      $
       \gm (\tau) 
        := 
       \LL_{\tau_1 , \tau} \exp_x (Z)
      $
    } 
    \\
    \mbox{ is a minimal $\LL$-geodesic} 
  \end{array}
  \right.
} 
, 
\\
\bar{\tau}( x , \tau ; Z ) 
& : = 
\sup
\bbra{
  \tau \in ( \tau_1 , T ) 
  \; | \; 
  Z \in \Omega (x , \tau_1 ; \tau )
}
.
\end{align*}
Based on these notations, 
we define the $\LL$-cut locus 
$\Lcut$ 
by 
\begin{equation*}
\Lcut 
: = 
\bbra{
  ( x , \tau_1 ; y , \tau_2 ) 
  \; \left| \; 
  \begin{array}{l}
    x \in M , \tau_1 \in [ \tauu_1 , T ) , 
    \\
    y = \LL_{\tau_1 , \tau_2} \exp_x ( Z ) 
    \mbox{ for some $Z \in T_x M$}, 
    \\
    \tau_2 = \bar{\tau} ( x, \tau_1 ; Z ) 
    \in ( \tau_1 ,  T )
  \end{array}
  \right.
}
. 
\end{equation*}
As remarked in \cite{chowetal,topping,ye}, 
$\Lcut$ is a union of two different kinds of sets. 
The first one consists of 
$(x, \tau_1 ; y, \tau_2)$ 
such that there exists more than one 
minimal $\LL$-geodesics 
joining $(x, \tau_1)$ and $(y, \tau_2)$.  
The second consists of 
$(x, \tau_1 ; y , \tau_2 )$ 
such that 
$(y, \tau_2)$ is conjugate to $(x, \tau_1)$ 
along a minimal $\LL$-geodesic 
with respect to $\LL$-Jacobi field. 
Note that we can define exponential map 
in the reverse direction in $\tau$. 
By using this reverse exponential map, 
``reversed $\LL$-cut locus'' is defined 
and it is identified with $\Lcut$ 
by virtue of 
the above characterization of $\Lcut$. 

\subsection{Proof of Theorem~\ref{mainresult}}
\label{sec:proof}

For the proof of our main theorem 
based on discrete approximation, 
we will follow a similar way 
as in previous studies in this direction 
(see \cite{kuwada,kuwada2} and references therein). 
Our first task is to show 
a difference inequality of 
$\Lambda ( t , \mathbf{X}^\eps_t )$ 
in Lemma~\ref{lem:2var-gRW}. 
We begin with introducing some notations. 
Set 
$
\gm_n 
: = 
\gm_{ \mathbf{X}^\eps_{ t_n } }^{ \tauu_1 t_n , \tauu_2 t_n}
$ and 
let us define a vector field $\hat{\lambda}_{n+1}^\dag$ 
along $\gm_n$ 
by 
$
\hat{\lambda}_{n+1} (\tau) 
 = 
\sqrt{\tau / t_n } \lambda_{n+1}^* (\tau)
$, 
where $\lambda_{n+1}^*$ is 
a space-time parallel vector field along $\gm_n$ 
with initial condition 
$
\hat{\lambda}_{n+1}^* ( \tauu_1 t_n ) 
= 
\hat{\lambda}_{n+1}^{(1)}
$. 
Let us define random variables 
$\zt_n$ and $\Sigma_n$ as follows: 
%
\begin{align*}
\zt_{n+1} 
& : = 
\sqrt{2 \tau} 
\left. 
\dbra{ 
  \hat{\lambda}_{n+1}^\dag (\tau) , 
  \dot{\gm}_n (\tau) 
}_{g(\tau)}
\right|_{\tau = \tauu_1 t_n}^{\tauu_2 t_n} 
, 
\\
\Sigma_{n+1} 
& := 
\frac{1}{t_n} 
\left. 
  \tau^{3/2} 
  \abra{ 
    R_{g(\tau)} ( \gm_n (\tau) ) 
    - 
    | \dot{\gm}_n (\tau) |_{g(\tau)}^2 
  }
\right|_{\tau = \tauu_1 t_n}^{\tauu_2 t_n} 
\\
& \qquad + 
\Bigg(
  \left. 
   \abra{ 
     \frac{\sqrt{\tau}}{t_n} 
     - 
     2 \sqrt{\tau} 
     \Ric_{g(\tau)} 
     ( 
      \hat{\lambda}_{n+1}^\dag (\tau), 
      \hat{\lambda}_{n+1}^\dag (\tau)
     )
   }
  \right|_{\tau=\tauu_1 t_n}^{\tauu_2 t_n} 
\\ 
& \hspace{18em} - 
  \int_{\tauu_1 t_n}^{\tauu_2 t_n} 
  \sqrt{\tau} 
  H \abra{ 
    \dot{\gm} (\tau) 
     , 
    \hat{\lambda}_{n+1}^\dag (\tau) 
  } 
  d \tau 
\Bigg)
.
\end{align*} 
Here $H$ is as given in \eqref{eq:H}.
For $M_0 \subset M$, 
we define  
$
\sg_{M_0} 
\: : \: 
C ( [ s , T/ \tauu_2 ] \to M \times M )  
\to 
[0, \infty )
$ 
by 
\begin{equation*}
\sg_{M_0} ( w, \tilde{w} )
: = 
\inf 
\bbra{ 
  t \ge s 
  \; | \; 
  w_{ \tauu_1 t } \notin M_0 
  \mbox{ or } 
  w_{ \tauu_2 t } \notin M_0 
}
. 
\end{equation*}
For simplicity of notations, 
$\sg_{M_0} ( \mathbf{X}^\eps )$ and 
$\sg_{M_0} ( \mathbf{X} )$ are denoted 
by $\sg_{M_0}^\eps$ and $\sg_{M_0}^0$ 
respectively. 
As shown in \cite{kuwada2}, 
for any $\eta > 0$, 
we can take a compact set $M_0 \subset M$ 
such that 
$
\lim_{\eps \to 0} 
\P [ \sg_{M_0}^\eps \le T ] \le \eta
$ 
holds (cf.~\cite{kuwadaphilipowski}). 
\begin{lemma} \label{lem:2var-gRW}
Let $M_0 \subset M$ be compact. 
Then 
there exist a family of random variables 
$( Q^\eps_n )_{ n \in \N , \eps > 0 }$ 
and a family of deterministic constants 
$( \delta (\eps) )_{\eps > 0}$ 
with $\lim_{\eps \to 0} \delta (\eps) = 0$ 
satisfying 
\begin{equation} \label{eq:error}
\sum_{ 
  n ; \; 
  t_n < \sg_{M_0}^\eps \wedge ( T / \tauu_2 ) 
} 
Q^\eps_n 
\le 
\delta (\eps) 
\end{equation}
such that 
\begin{equation} \label{eq:2var-gRW}
\Lambda ( t_{n+1} , \mathbf{X}^\eps_{ t_{n+1} } ) 
\le 
\Lambda ( t_n , \mathbf{X}^\eps_{ t_n } ) 
+ \eps \zt_{n+1}  
+ \eps^2 \Sigma_{n+1} 
+ Q^\eps_{n+1} 
. 
\end{equation}
\end{lemma}
\begin{proof}
When 
$( 
 X^\eps ( \tauu_1 t_n ) , \tauu_1 t_n 
 ; 
 Y^\eps ( \tauu_2 t_n ) , \tauu_2 t_n  
) 
\notin \Lcut
$, 
the inequality \eqref{eq:2var-gRW} 
follows from the Taylor expansion 
with the error term $Q^\eps_{n+1} = o (\eps^2)$. 
Indeed, 
the first variation formula 
(\cite[Lemma~7.15]{chowetal} cf.~\eqref{eq:first})
produces $\eps \zt_{n+1}$ and 
Corollary~\ref{hessianofl} together with \eqref{eq:dt} 
implies the bound $\eps^2 \Sigma_{n+1}$ 
of the second order term. 
To include the case 
$( 
 X^\eps ( \tauu_1 t_n ) , \tauu_1 t_n 
 ; 
 Y^\eps ( \tauu_2 t_n ) , \tauu_2 t_n  
) 
\in \Lcut
$ 
as well as 
to obtain a uniform bound \eqref{eq:error}, 
we extend this argument. 
Set $\tau_n^* := ( \tauu_1 + \tauu_2 ) t_n / 2$. 
Then we can show 
\begin{align*}
( 
 X^\eps_{ \tauu_1 t_n } , \tauu_1 t_n 
 ; 
 \gm_n ( \tau_{n}^* ) , \tau_n^* 
) 
& \notin \Lcut ,
\\
( 
 \gm_n ( \tau_{n}^* ) , \tauu_{n}^* ; 
 X^\eps_{ \tauu_2 t_n } , \tauu_2 t_n 
) 
& \notin \Lcut 
\end{align*} 
since minimal $\LL$-geodesics 
with these pair of endpoints 
can be extended with keeping its minimality 
(cf.~see \cite[Section 7.8]{chowetal} and \cite{ye}). 
Set 
$
x_{n+1}^* = 
  \exp_{\gm_n ( \tau_{n}^* )}^{(\tau_n^*)} 
  \abra{ 
    \sqrt{\tauu_1 + \tauu_2} 
    \lambda_{n+1}^\dag (\tau_n^*) 
  } 
$. 
The triangle inequality for $L$ yields  
\begin{align*}
\Lambda ( t_n , \mathbf{X}^\eps_{ t_n } ) 
& =  
L 
( 
 X^\eps_{ \tauu_1 t_n } , \tauu_1 t_n 
 ; 
 \gm_n ( \tau_{n}^* ) , \tau_{n}^* 
) 
+ 
L 
( 
 \gm_n ( \tau_{n}^* ) , \tau_{n}^* 
 ; 
 X^\eps_{ \tauu_2 t_n} , \tauu_2 t_n 
) 
, 
\\
\Lambda ( t_{n+1} , \mathbf{X}^\eps_{ t_{n+1} } ) 
& \le 
L 
\abra{ 
  X^\eps_{ \tauu_1 t_{n+1} } 
  , 
  \tauu_1 t_{n+1} 
  ; 
  x_{n+1}^*
  , 
  \tau_{n+1}^* 
} 
+ 
L 
\abra{ 
  x_{n+1}^* 
  , 
  \tau_{n+1}^* 
  ; 
  X^\eps_{ \tauu_2 t_{n+1}} 
  , 
  \tauu_2 t_{n+1} 
}
. 
\end{align*}
Hence 
\begin{align*}
\Lambda ( t_{n+1} , \mathbf{X}^\eps_{ t_{n+1} } ) 
- 
\Lambda ( t_n , \mathbf{X}^\eps_{ t_n } ) 
& \le
\abra{ 
  L 
  \abra{ 
    X^\eps_{ \tauu_1 t_{n+1} } 
    , 
    \tauu_1 t_{n+1} 
    ; 
    x_{n+1}^*
    , 
    \tau_{n+1}^* 
  } 
-
L 
( 
 X^\eps_{ \tauu_1 t_n } , \tauu_1 t_n 
 ; 
 \gm_n ( \tau_{n}^* ) , \tau_{n}^* 
) 
}
\\ & \quad 
+ 
\abra{
  L 
  \abra{ 
    x_{n+1}^* 
    , 
    \tau_{n+1}^* 
    ; 
    X^\eps_{ \tauu_2 t_{n+1}} 
    , 
    \tauu_2 t_{n+1} 
  }
  -
  L 
  ( 
  \gm_n ( \tau_{n}^* ) , \tau_{n}^* 
  ; 
  X^\eps_{ \tauu_2 t_n} , \tauu_2 t_n 
  ) 
}
\end{align*}
and 
the desired inequality with $Q_n^\eps = o ( \eps^2 )$ holds 
by applying the Taylor expansion 
to each term on the right hand side of 
the above inequality. 

We turn to showing 
the claimed control \eqref{eq:error} of 
the error term $Q^\eps_n$.  
Take $M_1 \supset M_0$ compact such that 
every minimal $\LL$-geodesic joining 
$( x , \tauu_1 t )$ and $( y , \tauu_2 t )$
is included in $M_1$ 
if $x,y \in M_0$ and $t \in [ s , T/ \tauu_2 ]$ .
Indeed, such $M_1$ exists 
since we have 
the lower bound of $L$ in \eqref{eq:L-bound} 
and $L$ is continuous. 
Let us define a set $A$ by  
\[
A : = 
\bbra{ 
  ( ( \tau_1 , x ) , ( \tau_3 , z ) , ( \tau_2 , y ) ) 
  \in 
  ( [ \tauu_1 , T ] \times M_1 )^3 
  \; \left| \;
      \begin{array}{l}
        x ,y \in M_0 , 
        \\
        \tau_2 - \tau_1 \ge ( \tauu_2 - \tauu_1 ) s ,
        \\
        \tau_3 = ( \tau_1 + \tau_2 ) / 2 ,
        \\
        L ( x , \tau_1 ; z , \tau_3 ) 
         + 
        L ( z , \tau_3 ; y , \tau_2 ) 
        \\
         = 
        L ( x , \tau_1 ; y , \tau_2 ) 
      \end{array}
  \right. 
}
. 
\]
Note that $A$ is compact. 
Let $\pi_1 , \pi_2 \: : \: A \to ( [ \tauu_1, T ] \times M_1 )^2$ be 
defined by 
\begin{align*}
\pi_1 ( ( \tau_1 , x ) , ( \tau_3 , z ) , ( \tau_2 , y ) ) 
& :=  
( ( \tau_1 , x ) , ( \tau_3 , z ) ) 
, 
\\
\pi_2 ( ( \tau_1 , x ) , ( \tau_3 , z ) , ( \tau_2 , y ) ) 
& := 
( ( \tau_3 , z ) , ( \tau_2 , y ) ) 
. 
\end{align*}
Then $\pi_1 (A)$ and $\pi_2 (A)$ are compact and 
$\pi_i (A) \cap \Lcut = \emptyset$ for $i=1,2$. 
The second assertion comes from the fact 
that $(z, \tau_3 )$ is on a minimal $\LL$-geodesic 
joining $( x, \tau_1 )$ and $( y, \tau_2 )$ 
for $( ( x, \tau_1 ) , ( z , \tau_3 ) , ( y , \tau_2 ) ) \in A$.  
Note that $\Lcut$ is closed 
(see \cite{topping}; 
though they assumed $M$ to be compact, 
an extension to the non-compact case is straightforward). 
Thus we can take relatively compact open sets 
$G_1 , G_2 \subset [ \tauu_1 , T ] \times M$ such that 
$\pi_i (H) \subset G_i$ and 
$\bar{G_i} \cap \Lcut = \emptyset$ for $i=1,2$. 
Then the Taylor expansion we discussed above can be done 
on $G_1$ or $G_2$ for sufficiently small $\eps$. 
Recall that $L$ is smooth outside of $\Lcut$ (see \cite{chowetal}). 
Thus the convergence $\eps^{-2} Q_n (\eps) \to 0$ as $\eps \to 0$ 
is uniform in $n$ and independent of $\mathbf{X}_{t_n}^\eps$ 
as long as $t_n < \sg_{M_0}^\eps \wedge ( T / \tauu_2 )$. 
Since the cardinality of 
$\{ n \; | \; t_n < \sg_{M_0}^\eps \wedge ( T / \tauu_2 ) \}$ 
is of order at most $\eps^{-2}$, 
the assertion \eqref{eq:error} holds. 
\end{proof} 
We next establish the corresponding difference inequality 
for $\Theta ( t , \mathbf{X}^\eps_t )$ 
(Corollary~\ref{cor:T2var-gRW}). 
For that, we show the following auxiliary lemma.  
\begin{lemma} \label{lem:var-bound} 
Let $T_0 < T$. 
\begin{enumerate}
\item[(i)]
There exist deterministic constants 
$c_2' > 0$ and $C_2' > 0$ 
such that 
\begin{align*} 
| \zt_{n} | 
& \le 
c_2' \rho_{g(T)} ( \mathbf{X}_{t_{n-1}}^\eps ) 
+ C_2' , 
&  
| \Lambda ( t_n , \mathbf{X}_{t_n}^\eps ) | 
& \le 
c_2' \rho_{g(T)} ( \mathbf{X}_{t_n}^\eps )^2 
+ C_2' 
\end{align*}
if $t_n \le T_0 / \tauu_2$. 
\item[(ii)]
Take $M_0 \subset M$ compact. 
Then there is a constant $R = R ( T_0 , M_0 ) >0$ 
such that $| \Sigma_n | \le R$ holds 
if $t_n < \sg_{M_0}^\eps \wedge ( T_0 / \tauu_2 )$.  
\end{enumerate}
\end{lemma}
\begin{proof}
By the definition of $\zt_{n}$, 
we have 
\begin{equation*} 
| \zt_{n} | 
\le 
\sqrt{2(d+2)} t_{n-1} 
\abra{ 
  \tauu_1 | \dot{\gm}_{n-1} ( \tauu_1 t_{n-1} ) |_{g(\tauu_1 t_{n-1})} 
  + 
  \tauu_2 | \dot{\gm}_{n-1} ( \tau_2 t_{n-1} ) |_{g(\tauu_2 t_{n-1})} 
}
. 
\end{equation*} 
Thus 
the desired bound for $| \zt_{n} |$ 
follows from 
\eqref{eq:velo-bound1} and \eqref{eq:metric-bound}. 
Similarly, 
the estimate 
for $\Lambda ( t_n , \mathbf{X}_{t_n}^\eps )$ 
follows from 
\eqref{eq:L-bound} and \eqref{eq:metric-bound}.    
For the assertion (ii), 
we deal with the integral involving $H$ 
in the definition of $\Sigma_n$. 
Note that 
every tensor field appeared in the definition of $H$ 
is continuous. 
As in the proof of Lemma~\ref{lem:2var-gRW}, 
take $M_1 \supset M_0$ compact such that 
every minimal $\LL$-geodesic joining 
$(x, \tauu_1 t)$ and $(y, \tauu_2 t)$ is 
included in $M_1$ if $x,y \in M_0$ and 
$t \in [ s , T /\tauu_2 ]$. 
Since 
$\mathbf{X}_{t_{n-1}}^\eps \in M_0 \times M_0$ holds 
on the event 
$\{ t_n < \sigma_{M_0}^\eps \wedge ( T_0 / \tauu_2 ) \}$, 
the upper bound \eqref{eq:velo-bound1} of 
$\sqrt{\tau} | \dot{\gm} (\tau) |$ 
implies that 
$H ( \dot{\gm}_n (\tau), Z (\tau) )$ is 
uniformly bounded 
for any vector field $Z (\tau)$ along $\gm_n$ 
of the form $Z (\tau) = \sqrt{ \tau / t_n} Z^* (\tau)$ 
with a space-time parallel vector field $Z^* (\tau)$ 
satisfying $| Z^* (\tau) |_{g (\tau)} \le 1$. 

This fact yields an expected bound for the integral. 
For any other terms in the definition of $\Sigma_n$, 
we can estimate them as in the assertion (i). 
\end{proof}
By virtue of Lemma~\ref{lem:var-bound}, 
$\Lambda ( t_n , \mathbf{X}_{t_n}^\eps )$, $\zt_n$ and $\Sigma_n$ 
are uniformly bounded 
on the event 
$\{ t_n < \sg_{M_0}^\eps \wedge ( T_0 / \tauu_2 ) \}$ 
for $T_0 < T$. 
Thus Lemma~\ref{lem:2var-gRW} yields the following: 
\begin{corollary} \label{cor:T2var-gRW}
Let $T_0 < T$ and $M_0 \subset M$ be a compact set. 
Then 
there exist a family of random variables 
$( \tilde{Q}^\eps_n )_{ n \in \N, \eps > 0 }$ 
and a family of deterministic constants 
$( \tilde{\delta} (\eps) )_{\eps > 0}$ 
with $\lim_{\eps \to 0} \tilde{\delta} (\eps) = 0$ 
satisfying 
\begin{equation*}
\sum_{ 
  n ; \; 
  t_n < \sg_{M_0}^\eps \wedge ( T_0 / \tauu_2 ) 
} 
\tilde{Q}^\eps_n  
\le 
\tilde{\delta} (\eps) 
\end{equation*}
such that 
\begin{align} \nonumber
\Theta ( t_{n+1} , \mathbf{X}^\eps_{ t_{n+1} } ) 
& \le 
\Theta ( t_n , \mathbf{X}^\eps_{ t_n } ) 
 + 
\frac{\eps^2}{\sqrt{t_n}}
\abra{ \sqrt{\tauu_2} - \sqrt{\tauu_1} } 
\Lambda ( t_n , \mathbf{X}^\eps_{t_n} )
 - 
2 \eps^2 d 
\abra{ \sqrt{\tauu_2 } - \sqrt{\tauu_1 } }^2  
\\ \nonumber
& \qquad + 
2 \eps 
\abra{ \sqrt{\tauu_2 t_{n+1}} - \sqrt{\tauu_1 t_{n+1}} } 
\zt_{n+1}  
 + 
2 \eps^2 
\abra{ \sqrt{\tauu_2 t_{n+1}} - \sqrt{\tauu_1 t_{n+1}} } 
\Sigma_{n+1} 
\\ \label{eq:T2var-gRW}
& \qquad + 
\tilde{Q}^\eps_{n+1} 
. 
\end{align}
\end{corollary}
The term $\Sigma_n$ corresponds to 
the one dominating the bounded variation part 
of $d \Lambda ( t , \tilde{X}_t , \tilde{Y}_t )$ 
in section~\ref{sec:SDE}. 
However, as a result of our discretization, 
we are no longer able to apply 
Proposition~\ref{lambda} directly 
to estimate $\Sigma_n$ itself. 
In this case, we can do it 
to the conditional expectation of $\Sigma_n$ instead. 
Set $\mathcal{G}_n := \sg ( \lambda_1 , \ldots , \lambda_n )$ 
and $\bar{\Sigma}_{n+1} := \E [ \Sigma_{n+1} \: | \: \mathcal{G}_n ]$. 
Then, since 
each $\Phi_i$ is isometry 
and 
$
(d+2) 
\E [ 
 \dbra{ \lambda_n , e_i } 
 \dbra{ \lambda_n , e_j } 
] 
= 
\dl_{ij}
$, 
Proposition~\ref{lambda} yields 
\begin{equation} \label{eq:expected}
\bar{\Sigma}_n 
\le 
\frac{d}{\sqrt{t_n}} 
\abra{ 
  \sqrt{\tauu_2} - \sqrt{\tauu_1}
} 
- 
\frac{1}{2 t_n} 
\Lambda ( t_n , \mathbf{X}^\eps_{ t_n } ) 
. 
\end{equation} 
In order to replace $\Sigma_n$ with $\bar{\Sigma}_n$ 
in \eqref{eq:T2var-gRW}, 
we show the following: 
\begin{lemma} \label{lem:Doob}
Let $t_0 < T_0 < T$ and $M_0 \subset M$ compact.  
For $t \in [ \tauu_1 , T ]$, 
set 
$
N_t^\eps 
:= 
\sup \{ n \in \N \; | \; \tauu_2 ( s + \eps^2 n ) \le t \}
$. 
Then, for $\eta > 0$, 
\begin{equation*} 
\limsup_{\eps \to 0} 
\P \cbra{ 
  \sup_{
    \begin{subarray}{c}
        N_{t_0}^\eps \le N \le N_{T_0}^\eps 
        \\ 
        t_N \le \sg_{M_0}^\eps 
    \end{subarray}
  } 
  \left| 
    \sum_{n=1}^{N} \sqrt{t_n} ( \Sigma_n - \bar{\Sigma}_n )
  \right| 
  > \eps^{-2} \eta 
} 
= 0. 
\end{equation*}
\end{lemma}
\begin{proof}
Lemma~\ref{lem:var-bound} ensures that 
$\Sigma_n$ and $\bar{\Sigma}_n$ is bounded 
as long as 
$n \le N_{T_0}^\eps$ and 
$t_n <  \sg_{M_0}^\eps$. 
Thus 
$\sum_{n=1}^N \sqrt{t_n} ( \Sigma_n - \bar{\Sigma}_n )$ 
is a $\mathcal{G}_N$-local martingale. 
Hence the Doob inequality implies 
\begin{equation} \label{eq:Doob}
\limsup_{\eps \to 0} 
\P \cbra{ 
  \sup_{
    \begin{subarray}{c}
        0 \le N \le N_{t}^\eps 
        \\ 
        t_N \le \sg_{M_0}^\eps 
    \end{subarray}
  } 
  \left| 
    \sum_{n=1}^{N} \sqrt{t_n} ( \Sigma_n - \bar{\Sigma}_n )
  \right| 
  > \eps^{-2} \eta 
} 
= 0 
\end{equation}
for $t \in [ \tauu_1 , T_0 ]$. 
Since we have 
\begin{align*} 
& 
\bbra{ 
  \sup_{
    \begin{subarray}{c}
        N_{t_0}^\eps \le N \le N_{T_0}^\eps 
        \\ 
        t_N \le \sg_{M_0}^\eps 
    \end{subarray}
  } 
  \left| 
    \sum_{ n = N_{t_0}^\eps }^{N} 
    \sqrt{t_n} ( \Sigma_n - \bar{\Sigma}_n )
  \right| 
  > \eps^{-2} \eta 
} 
\\
& \qquad \subset 
\bbra{ 
  \sup_{
    \begin{subarray}{c}
        0 \le N \le N_{t_0}^\eps 
        \\ 
        t_N \le \sg_{M_0}^\eps 
    \end{subarray}
  } 
  \left| 
    \sum_{n=1}^{N} \sqrt{t_n} ( \Sigma_n - \bar{\Sigma}_n )
  \right| 
  > \frac{\eps^{-2} \eta }{2}
} 
\cup 
\bbra{ 
  \sup_{
    \begin{subarray}{c}
        0 \le N \le N_{T_0}^\eps 
        \\ 
        t_N \le \sg_{M_0}^\eps 
    \end{subarray}
  } 
  \left| 
    \sum_{n=1}^{N} \sqrt{t_n} ( \Sigma_n - \bar{\Sigma}_n )
  \right| 
  > \frac{\eps^{-2} \eta}{2} 
}, 
\end{align*}
the assertion follows from \eqref{eq:Doob}. 
\end{proof}
As a final preparation, 
we show the following auxiliary lemma. 
\begin{lemma} \label{lem:i'bility}
There exists $C_3 > 0$ such that 
\begin{equation*} 
\E 
\cbra{ 
  \sup_{s \le t \le T / \tauu_2 } 
  | \Theta ( t , \mathbf{X}_t ) |^2  
} 
< C_3 
.  
\end{equation*} 
\end{lemma}
\begin{proof}
By virtue of \eqref{eq:L-bound}, 
$\Theta ( t , \mathbf{X}_t )$ 
is bounded from below 
uniformly in $t \in [ s , T / \tauu_2 ]$. 
In addition, 
there is a constant $c, C > 0$ such that 
\begin{equation*}
\Theta ( t , \mathbf{X}_t ) 
\le 
c \rho_{g(T)} ( \mathbf{X}_t )^2 + C
\end{equation*} 
holds 
for $t \in [ s , T / \tauu_2 ]$. 
Take a reference point $o \in M$. 
Then we have 
\begin{equation*}
\rho_{g(T)} ( \mathbf{X}_t ) 
\le 
\e^{C_0 T } 
\abra{ 
  \rho_{g (\tau_1 t)} ( o , X_{\tau_1 t} ) 
  + 
  \rho_{g(\tau_2 t)} ( o , Y_{\tau_2 t} )
}. 
\end{equation*}
Thus the proof can be reduced to the following claim:
\begin{equation}\label{eq:moment}
\E \cbra{ 
  \sup_{
    \tauu_2 s \le t \le T 
  } 
  \rho_{g(\tauu_2 t)} ( o , Y_{\tauu_2 t} )^4
} < \infty 
.
\end{equation}
Indeed, a similar bound 
for $X_{ \tauu_1 t}$ follows 
in the same way. 
As shown in \cite{kuwadaphilipowski}, 
$( \rho_{g(t)} ( o , Y_t ) )_{[\tauu_2 s , T ]}$ 
is dominated from above 
by a Bessel process 
(of dimension $d$) 
plus a constant. 
Thus \eqref{eq:moment} easily follows from 
the Burkholder inequality for the fourth moment of 
a Euclidean Brownian motion. 
\end{proof}
%
%
%
\begin{proof}[Proof of Theorem~\ref{mainresult}]
%
First we remark that 
the map $(x,y) \mapsto ( X^\a , Y^\a )$ is 
obviously measurable. 
Thus, 
we obtain the same measurability for $( X, Y )$. 
The integrability of $\Theta ( t , \mathbf{X}_t )$ 
follows from Lemma~\ref{lem:i'bility}. 
We will show the supermartingale property 
in the sequel. 
For $s \le s_1 < \cdots < s_m < t' < t < T$ 
and 
$
f_1 , \ldots , f_m 
\in 
C_c (M \times M \to \R )
$ with $0 \le f_j \le 1$, 
Set 
$
F ( \mathbf{w} ) 
 := 
\prod_{j=1}^m 
f_j ( \mathbf{w}_{s_j} )
$
for 
$\mathbf{w} \in  C ( [ s , T / \tauu_2 ] \to M \times M)$.
Take $\eta > 0$ arbitrarily and 
choose a relatively compact open set 
$M_0 \subset M$ 
so that 
$
\P [ \sg_{M_0}^0 \le t ] 
\le 
\eta
$ 
holds. 
Note that 
$
\limsup_{\eps \to 0} 
\P [ \sg_{M_0}^\eps \le t ] 
\le 
\eta
$ 
also holds 
since $\{ w \; | \; \sg_{M_0} (w) \le t \}$ is closed.  
It suffices to show that 
there is a constant $C > 0$ which is 
independent of $\eta$ and $M_0$ 
such that, 
\begin{equation} \label{eq:sup-mart} 
\E \cbra{ 
  \abra{  
    \Theta 
    (
     t \wedge \sg_{M_0}^0 , 
     \mathbf{X}_{ t \wedge \sg_{M_0}^0 } 
    )
     - 
    \Theta 
    (
     t' \wedge \sg_{M_0}^0 , 
     \mathbf{X}_{ t' \wedge \sg_{M_0}^0 }
    )
  } 
  F ( \mathbf{X}_{ \cdot \wedge \sg_{M_0}^0 } )
} 
\le C \sqrt{\eta}   
\end{equation}
holds. 
In fact, once we have shown \eqref{eq:sup-mart}, 
then
Lemma~\ref{lem:i'bility} yields 
\begin{equation*}
\E \cbra{ 
  \abra{  
    \Theta 
    (
     t , 
     \mathbf{X}_t
    )
     - 
    \Theta 
    (
     s ,
     \mathbf{X}_s
    )
  } 
  F ( \mathbf{X} ) 
} 
\le 0 
\end{equation*}
since $\sg_{M_0}^0 \to \infty$ 
almost surely as $M_0 \uparrow M$. 

Take $f \in C_c ( M \times M )$ 
such that 
$0 \le f \le 1$ and $f|_{U} \equiv 1$, 
where $U \subset M \times M$ is a open set 
containing $\bar{M}_0 \times \bar{M}_0$. 
Then, 
by virtue of 
Lemma~\ref{lem:i'bility} 
and 
the choice of $M_0$, 
\begin{multline} \label{eq:discretize0}
\E  \Big[ 
  \Big(
  \Theta 
   (
    t \wedge \sg_{M_0}^0 , 
    \mathbf{X}_{ t \wedge \sg_{M_0}^0 } 
   )
   - 
  \Theta 
   (
    t' \wedge \sg_{M_0}^0 , 
    \mathbf{X}_{ t' \wedge \sg_{M_0}^0 }
   )
  \Big) 
  F ( \mathbf{X}_{ \cdot \wedge \sg_{M_0}^0 } )
\Big] 
\\ 
\le 
\E \cbra{ 
  \abra{ 
    \Theta 
    (
     t , 
     \mathbf{X}_{t}
    )
     - 
    \Theta 
    (
     t' , 
     \mathbf{X}_{ t' }
    )
  } 
  f ( \mathbf{X}_t ) f (\mathbf{X}_{t'}) 
  F ( \mathbf{X} ) 
  \; ; \; \sg_{M_0}^0 > t 
} 
+ 2 C_3^{1/2} \sqrt{\eta}
.
\end{multline}
For $u \in [ s , T /\tauu_2 ]$, 
let us define $\ebra{ u }_\eps$ 
by 
\begin{equation*}
\ebra{ u }_\eps 
: = 
\sup 
\{ s + \eps^2 n 
  \; | \; 
  n \in \N \cup \{ 0 \}, 
  1 + \eps^2 n < u 
\}
. 
\end{equation*}
Then, 
since $\{ w \; | \; \sg_{M_0} (w) > t \}$ is open,  
\begin{align} \nonumber
\E 
& 
\cbra{ 
  \abra{ 
    \Theta 
    (
     t , 
     \mathbf{X}_{t}
    )
     - 
    \Theta 
    (
     t' , 
     \mathbf{X}_{ t' }
    )
  } 
  f ( \mathbf{X}_t ) 
  f ( \mathbf{X}_{t'} ) 
  F ( \mathbf{X} ) 
  \; ; \; \sg_{M_0}^0 > t 
} 
\\ \nonumber
& \le 
\liminf_{\eps \to 0}
\E \cbra{ 
  \abra{ 
    \Theta 
    (
     t , 
     \mathbf{X}_{t}^\eps
    )
     - 
    \Theta 
    (
     t' , 
     \mathbf{X}_{t'}^\eps
    )
  } 
  f ( \mathbf{X}_t^\eps ) 
  f ( \mathbf{X}_{t'}^\eps ) 
  F ( \mathbf{X}^\eps ) 
  \; ; \; \sg_{M_0}^\eps > t 
} 
\\ \label{eq:discretize1}
& = 
\liminf_{\eps \to 0}
\E \cbra{ 
  \abra{ 
    \Theta 
    (
     \ebra{t}_\eps , 
     \mathbf{X}_{ \ebra{t}_\eps }^\eps
    )
     - 
    \Theta 
    (
     \ebra{t'}_\eps , 
     \mathbf{X}_{ \ebra{t'}_\eps }^\eps
    )
  } 
  f ( \mathbf{X}_{ \ebra{t}_\eps }^\eps ) 
  f ( \mathbf{X}_{ \ebra{t'}_\eps }^\eps ) 
  F ( \mathbf{X}^\eps ) 
  \; ; \; \sg_{M_0}^\eps > t 
} 
.
\end{align}
Here the last inequality follows 
from the continuity of $\Theta$ and $f$. 
Set 
$\hat{\sg}_{M_0}^\eps := \ebra{ \sg_{M_0}^\eps }_\eps + \eps^2$. 
Note that $\{ \hat{\sg}_{M_0}^\eps = t_n \} \in \mathcal{G}_n$ 
for all $n \in \N \cup \{ 0 \}$. 
Then 
\begin{align} \nonumber 
\E & 
\cbra{ 
  \abra{ 
    \Theta 
    (
     \ebra{t}_\eps , 
     \mathbf{X}_{ \ebra{t}_\eps }^\eps
    )
     - 
    \Theta 
    (
     \ebra{t'}_\eps , 
     \mathbf{X}_{ \ebra{t'}_\eps }^\eps
    )
  } 
  f ( \mathbf{X}_{ \ebra{t}_\eps }^\eps ) 
  f ( \mathbf{X}_{ \ebra{t'}_\eps }^\eps ) 
  F ( \mathbf{X}^\eps ) 
  \; ; \; \sg_{M_0}^\eps > t 
} 
\\ \nonumber
& \hspace{4em} \le 
\E \cbra{ 
  \abra{ 
    \Theta 
    (
     \ebra{t}_\eps \wedge \hat{\sg}_{M_0}^\eps , 
     \mathbf{X}_{ 
       \ebra{t}_\eps \wedge \hat{\sg}_{M_0}^\eps 
     }^\eps
    )
     - 
    \Theta 
    (
     \ebra{t'}_\eps \wedge \hat{\sg}_{M_0}^\eps , 
     \mathbf{X}_{ 
       \ebra{t'}_\eps \wedge \hat{\sg}_{M_0}^\eps 
     }^\eps
    )
  } 
  F 
  ( 
   \mathbf{X}_{ 
     \cdot \wedge \hat{\sg}_{M_0}^\eps 
   }^\eps 
  ) 
} 
\\ \label{eq:discretize2}
& \hspace{12em} + 
2 
\E \cbra{ 
  \sup_{
    s \le u \le T / \tauu_2
  } 
  \left| 
      \Theta ( u , \mathbf{X}_u^\eps ) 
      f ( \mathbf{X}_u^\eps ) 
  \right|^2 
}^{1/2} 
\P [ \sg_{M_0}^\eps \le t ]^{1/2}
. 
\end{align} 
Since a function 
$
\mathbf{w} \mapsto \sup_{1 \le u \le T / \tauu_2} 
| \Theta ( u, \mathbf{w}_u ) f ( \mathbf{w}_u ) |
$ 
on 
$C ( [ s, T / \tauu_2 ] \to M \times M )$ 
is bounded and continuous, 
we have 
\begin{equation} \label{eq:discretize3}
\limsup_{\eps \to 0} 
\E \cbra{ 
  \sup_{
    s \le u \le T / \tauu_2
  } 
  \left| 
      \Theta ( u , \mathbf{X}_u^\eps ) 
      f ( \mathbf{X}_u^\eps ) 
  \right|^2 
}^{1/2} 
\P [ \sg_{M_0}^\eps \le t ]^{1/2}
\le 
C_3^{1/2} \sqrt{\eta} 
. 
\end{equation}
By combining 
\eqref{eq:discretize1}, 
\eqref{eq:discretize2} and 
\eqref{eq:discretize3} 
with \eqref{eq:discretize0}, 
the proof of \eqref{eq:sup-mart} is reduced to 
show the following estimate: 
\begin{equation} \label{eq:sup-mart1}
\limsup_{\eps \to 0} 
\E \cbra{ 
  \abra{ 
    \Theta 
    (
     \ebra{t}_\eps \wedge \hat{\sg}_{M_0}^\eps , 
     \mathbf{X}_{ 
       \ebra{t}_\eps \wedge \hat{\sg}_{M_0}^\eps 
     }^\eps
    )
     - 
    \Theta 
    (
     \ebra{t'}_\eps \wedge \hat{\sg}_{M_0}^\eps , 
     \mathbf{X}_{ 
       \ebra{t'}_\eps \wedge \hat{\sg}_{M_0}^\eps 
     }^\eps
    )
  } 
  F 
  ( 
   \mathbf{X}_{
     \cdot \wedge \hat{\sg}_{M_0}^\eps 
   }^\eps 
  ) 
} 
\le C \sqrt{\eta}
. 
\end{equation} 
%
%
%
%
Take $N_1$ and $N_2$  so that  
$t_{N_1} = \ebra{ t' }_\eps \wedge \hat{\sg}_{M_0}^\eps$ and 
$t_{N_2} = \ebra{ t }_\eps \wedge \hat{\sg}_{M_0}^\eps$ hold. 
Let $E_\eta$ be an event defined by 
\begin{equation*}
E_\eta 
: = 
\bbra{ 
  \left| 
      \sum_{n = N_1}^{N_2} 
      \sqrt{t_n} ( \Sigma_n - \bar{\Sigma}_n )
  \right| 
  \le 
  \frac{\sqrt{\eta}}
       {2 \eps^2 ( \sqrt{\tauu_2} - \sqrt{\tauu_1} ) } 
}
.
\end{equation*}
On $E_\eta$, 
an iteration of \eqref{eq:T2var-gRW} 
together with \eqref{eq:expected} 
yields 
\begin{equation*}
\Theta ( t_{N_2} , \mathbf{X}_{t_{N_2}} ) 
- 
\Theta ( t_{N_1} , \mathbf{X}_{t_{N_1}} ) 
\le 
\eps \sum_{n = N_1 +1}^{N_2} 
\abra{ 
  \sqrt{ \tauu_2 t_{n} } 
  -
  \sqrt{ \tauu_1 t_{n} }
} \zt_n
+ 2 \sqrt{\eta}
\end{equation*}
for sufficiently small $\eps$. 
In addition, Lemma~\ref{lem:Doob} yields 
$\limsup_{\eps \to 0} \P [ E_\eta^c ] = 0$. 
By applying Lemma~\ref{lem:var-bound} with $T_0 = \tauu_2 t$ 
to an iteration of \eqref{eq:T2var-gRW}, 
we obtain a constant $C > 0$ satisfying 
\[
\left| 
    \Theta ( t_{N_2} , \mathbf{X}_{t_{N_2}}^\eps ) 
    - 
    \Theta ( t_{N_1} , \mathbf{X}_{t_{N_1}}^\eps ) 
    - 
    \eps
    \sum_{n = N_1 + 1}^{N_2} 
    \abra{ 
      \sqrt{ \tauu_2 t_n } 
      - 
      \sqrt{ \tauu_1 t_n } 
    }
    \zt_n  
\right| < C 
\]
uniformly in sufficiently small $\eps > 0$. 
Since 
$F ( \mathbf{X}_{\cdot \wedge \hat{\sg}_{M_0}}^\eps )$ is 
$\mathcal{G}_{N_1}$-measurable, 
we obtain 
\begin{align*} \nonumber
\E & \cbra{ 
  ( 
   \Theta ( t_{N_2} , \mathbf{X}_{t_{N_2}}^\eps ) 
    - 
   \Theta ( t_{N_1} , \mathbf{X}_{t_{N_1}}^\eps ) 
  )
  F ( \mathbf{X}_{ \cdot \wedge \hat{\sg}_{M_0}^\eps }^\eps )
} 
\le 
C \, 
\P \cbra{ E_\eta^c }^{1/2} 
 + 
2 \sqrt{\eta} 
. 
\end{align*}
Hence \eqref{eq:sup-mart1} holds with $C = 1/2$ 
and the proof is completed.
\end{proof}

\end{document}